%% file: F_multi2.tex
\documentclass{amsart}

\usepackage{sty/preamble}

\newcommand{\onen}{{1 \cdots n}}
\newcommand{\jonejn}{{j_1, \ldots, j_n}}
\newcommand{\konekn}{{k_1, \ldots, k_n}}
\newcommand{\elloneelln}{{\ell_1, \ldots, \ell_n}}

\newcommand{\bigf}{\mathbf{F}}
\newcommand{\bigfbar}{\overline{\bigf}}
\newcommand{\bige}{\mathbf{E}}
\newcommand{\tail}{\bullet}
\newcommand{\whisk}{\textup{\dagger}}

\crefformat{enumi}{(#2#1#3)} 
\crefmultiformat{enumi}{(#2#1#3)}%
{ and~(#2#1#3)}{,~(#2#1#3)}{ and~(#2#1#3)}
\crefrangeformat{enumi}{(#3#1#4)\crefrangeconjunction(#5#2#6)}

\newcommand{\zer}{\texttt{0}} 

\title[Homotopy Equivalent Algebraic Structures]{Homotopy Equivalent Algebraic Structures in Multicategories and Permutative Categories}

\authorinfoNJDY

\hypersetup{pdfauthor=\authors}

\date{04 October 2022}

\subjclass[2020]{Primary: 18M65; Secondary: 55P42, 18M05}
\keywords{multicategory, permutative category, homotopy equivalence, operad algebra}

\begin{document}

\begin{abstract}
  \input{F_multi/abstract.tex}
\end{abstract}

\maketitle

\tableofcontents
\input{F_multi/content2.tex}

\bibliographystyle{sty/amsalpha3}
\bibliography{references}
\end{document}

%% file: F_multi/abstract.tex
We show that the free construction from multicategories to permutative categories is a categorically-enriched non-symmetric multifunctor.  Our main result then shows that the induced functor between categories of algebras is an equivalence of homotopy theories.  We describe an application to ring categories.

%% file: F_multi/content2.tex
\section{Introduction}\label{sec:intro}

The endomorphism construction provides a $\Cat$-enriched multifunctor
\[
  \End\cn \permcatsu \to \Multicat
\]
where $\permcatsu$ is the $\Cat$-enriched multicategory of small permutative categories, multilinear functors, and multilinear transformations, and $\Multicat$ is the $\Cat$-enriched multicategory of small multicategories, multifunctors, and multinatural transformations.  Thus it induces a functor on categories of algebras
\[
  \End^\cO \cn (\permcatsu)^\cO \to \Multicat^\cO
\]
for any small $\Cat$-enriched multicategory $\cO$.

Regarded as a 2-functor of underlying 2-categories, and restricting to the sub-2-category $\permcatst$ of strict symmetric monoidal functors, $\End$ has a left 2-adjoint $\bigf$ described in \cite[Theorem~4.2]{elmendorf-mandell-perm} and \cite[Section~5]{johnson-yau-permmult}.  With respect to the stable equivalences created by Segal's $K$-theory functor \cite{segal}, the adjoint pair $(\bigf,\End)$ is shown to be an equivalence of homotopy theories in \cite[Theorem~7.3]{johnson-yau-permmult}.

In this article we show that $\bigf$ extends to a non-symmetric $\Cat$-enriched multifunctor taking values in $\permcatsu$.  This extension, also denoted $\bigf$, fails to be adjoint to $\End$; see \cref{remark:epz} for further discussion of this detail.

Our main result shows, nevertheless, that the induced functors between categories of $\cO$-algebras do induce an equivalence of homotopy theories, for a non-symmetric $\Cat$-multicategory $\cO$.  This is stated as follows.
\begin{theorem}\label{theorem:alg-hty-equiv}
  Suppose $\cO$ is a small non-symmetric $\Cat$-multicategory.  Then $\bigf$ and $\bige = \End$ induce an equivalence of homotopy theories between categories of non-symmetric algebras
  \[
    \bigf^\cO \cn \big(\Multicat^\cO, \cS\big) \lradj \big((\permcatsu)^\cO, \cS \big) \cn \bige^\cO,
  \]
  where, in each category, $\cS$ denotes the class of componentwise stable equivalences.
\end{theorem}

The proof of \cref{theorem:alg-hty-equiv} is given in \cref{sec:hty-thy-equiv}.
Our main application, \Cref{corollary:En-alg-hty-equiv} focuses on ring categories. 
One key advantage of $\Multicat$ over $\permcatsu$ is that the former has a closed symmetric monoidal structure provided by the Boardman-Vogt tensor product (\cref{definition:BVtensor}).
The related smash product for pointed multicategories gives the $\Cat$-multicategory structure on $\permcatsu$, but does not induce a symmetric monoidal structure (see \cite[5.7.23 and~10.2.17]{cerberusIII}).

\subsection*{Outline}
We begin with several sections reviewing relevant concepts.
In  \cref{sec:hty-thy} we review equivalences of homotopy theories from \cite{gjo1}.
In \cref{sec:enrmulticat,sec:Multicat,sec:permcat} we review enriched multicategories and the $\Cat$-multicategory of permutative categories.  The third of these sections recalls the $\Cat$-multifunctor $\End$ from \cite{cerberusIII}.
In \cref{sec:free-perm} we review the definition of the left adjoint $\bigf$ from \cite{elmendorf-mandell,johnson-yau-permmult}.

The main results of this article are contained in the remaining sections.
We show that $\bigf$ is a non-symmetric $\Cat$-multifunctor in \Cref{sec:multimorphism-assignment,sec:cat-multi-f}.
In \cref{sec:two-transf} we develop the transformations comparing the composites $\bige\bigf$ and $\bigf\bige$ with the respective identity functors.
The proof of \cref{theorem:alg-hty-equiv} appears in \cref{sec:hty-thy-equiv}, after a review of stable equivalences for multicategories and permutative categories.
\Cref{sec:application} describes the application to ring categories as $E_1$-algebras in $\Cat$.

\subsection*{Acknowledgment}
We thank the referee for several helpful suggestions.

\section{Equivalences of Homotopy Theories}\label{sec:hty-thy}

In this section we review the theory of complete Segal spaces due to Rezk \cite{rezk-homotopy-theory}.
An equivalence of homotopy theories (\cref{definition:hty-thy}) is an equivalence of fibrant replacements in the complete Segal space model structure.
For further context and development we refer the reader to \cite{dwyer-kan,hirschhorn,toen-axiomatisation,barwick-kan}.

\subsection*{Complete Segal Spaces}

For the purpose of this paper, we only need to know the existence of the complete Segal space model structure in \cref{theorem:css-fibrant}; we will not use the specific definition of a complete Segal space.  The next definition is included only for the reader's information.  A \emph{bisimplicial set} is a simplicial object in the category of simplicial sets.  For a bisimplicial set $X = \{X(n)\}_{n \geq 0}$, each object $X(n)$ is a simplicial set.  In the following definition, $\mathbf{2}$ denotes the nerve, also known as the classifying space, of the category consisting of two isomorphic objects.  See \cite[Example I.1.4]{goerss-jardine} or \cite[Definition 7.2.3]{cerberusIII} for more discussion of the nerve.  

\begin{definition}\label{definition:css}
  A bisimplicial set $X$ is a \emph{complete Segal space} if it satisfies the following three conditions.
  \begin{enumerate}
  \item $X$ is a fibrant object in the Reedy model structure on bisimplicial sets.
  \item For each $n \ge 2$ the Segal map
    \[
    X(n) \to X(1) \times_{X(0)} \cdots \times_{X(0)} X(1)
    \]
    is a weak equivalence of simplicial sets.
  \item The morphism
    \begin{equation}\label{eq:css-char}
    X(0) \iso \Map(\De[0],X) \to \Map(\mathbf{2},X),
    \end{equation}
    which is induced by the unique morphism $\mathbf{2} \to \De[0]$, is a weak equivalence of simplicial sets. \dqed
  \end{enumerate}
\end{definition}

\begin{remark}
  The definition of complete Segal space given above is equivalent to that given in \cite[Section~6]{rezk-homotopy-theory} by \cite[6.4]{rezk-homotopy-theory}.
\end{remark}

\begin{theorem}[{\cite[7.2]{rezk-homotopy-theory}}]\label{theorem:css-fibrant}
  There is a simplicial closed model structure on the category of bisimplicial sets, called the \emph{complete Segal space model structure}, that is given as a left Bousfield localization of the Reedy model structure and in which the fibrant objects are precisely the complete Segal spaces.
\end{theorem}

\subsection*{Relative Categories}

\begin{definition}\label{definition:rel-cat}
  A \emph{relative category} is a pair $(\C,\cW)$ consisting of a category $\C$ and a subcategory $\cW$ containing all of the objects of $\C$.
  A \emph{relative functor}
  \[
  F\cn (\C,\cW) \to (\C',\cW')
  \]
  is a functor from $\C$ to  $\C'$ that sends morphisms of $\cW$ to those of $\cW'$.
\end{definition}

\begin{definition}\label{definition:rel-cat-pow}
  Suppose $(\C,\cW)$ is a relative category and $\A$ is another category.
  We let
  \[
  (\C,\cW)^\A
  \]
  denote the subcategory of $\C^\A$ whose objects are functors $\A \to \C$ and whose morphisms are those natural transformations with components in $\cW$.
\end{definition}

\begin{definition}\label{definition:classification-diagram}
  Suppose $(\C,\cW)$ is a relative category.
  The \emph{classification diagram} of $(\C,\cW)$ is the bisimplicial set
  \[
  \Ncl(\C,\cW) = \Ner\big((\C,\cW)^{\De[\bdot]}\big)
  \]
  given by
  \[
  n \mapsto \Ner\big((\C,\cW)^{\De[n]}\big)
  \]
  where $\De[n]$ denotes the category consisting of $n$ composable arrows.
\end{definition}

\begin{definition}\label{definition:hty-thy}
  Suppose $(\C,\cW)$ is a relative category.  We say that a bisimplicial set $R\Ncl(\C,\cW)$ is a \emph{homotopy theory of $(\C,\cW)$} if it is a fibrant replacement of $\Ncl(\C,\cW)$ in the complete Segal space model structure.
  We say that a relative functor
  \[
  F \cn (\C,\cW) \to (\C',\cW')
  \]
  is an \emph{equivalence of homotopy theories} if the induced morphism $R\Ncl F$ between homotopy theories is a weak equivalence in the complete Segal space model structure.
\end{definition}

\begin{remark}
  For readers familiar with the notions of hammock localization and $DK$-equivalence \cite{dwyer-kan}, Barwick and Kan have shown in \cite[1.8]{barwick-kan} that a relative functor
  \[
  F \cn (\C,\cW) \to (\C',\cW')
  \]
  is an equivalence of homotopy theories if and only if it induces a $DK$-equivalence between hammock localizations.
  In that case, $F$ induces equivalences between mapping simplicial sets and between categories of components.
  In particular, if $F$ is an equivalence of homotopy theories then the induced functor between categorical localizations
  \[
  \C[\cW^\inv] \to \C'[(\cW')^\inv]
  \]
  is an equivalence.
\end{remark}

\begin{proposition}[{\cite[2.8]{gjo1}}]\label{gjo28}
  Suppose
  \[
  F \cn (\C,\cW) \to (\C',\cW')
  \]
  is a relative functor and suppose that $F$ induces a weak equivalence of simplicial sets
  \begin{equation}\label{eq:F-level-we}
  \Ner\big((\C,\cW)^{\De[n]}\big) \to \Ner\big((\C',\cW')^{\De[n]}\big)
  \end{equation}
  for each $n$.  Then $F$ is an equivalence of homotopy theories.
\end{proposition}
\begin{proof}
  The assumption that \cref{eq:F-level-we} is a weak equivalence for each $n$ means that
  \[
  \Ncl F\cn \Ncl(\C,\cW) \to \Ncl(\C',\cW')
  \]
  is a weak equivalence between classification diagrams in the Reedy model structure \cite[Section~2.4]{rezk-homotopy-theory}.
  Thus $\Ncl F$ is a weak equivalence in the complete Segal space model structure because it is a localization of the Reedy model structure.
  As a consequence, $\Ncl F$ induces a weak equivalence between the homotopy theories given by fibrant replacements.
\end{proof}

Because a natural transformation between functors induces a simplicial homotopy on nerves, we have the following application of \cref{gjo28}.  Its proof is similar to that of \cite[2.12]{johnson-yau-permmult}.
\begin{proposition}[{\cite[2.9]{gjo1}}]\label{gjo29}
  Suppose given relative functors
  \[
  F \cn (\C, \cW) \lradj (\C',\cW') \cn E
  \]
  such that each of the composites $EF$ and $FE$ is connected to the respective identity functor by a zigzag of natural transformations whose components are in $\cW$ and $\cW'$, respectively.
  Then $F$ and $E$ are equivalences of homotopy theories.
\end{proposition}

\section{Enriched Multicategories}\label{sec:enrmulticat}

In this section we review the definitions of enriched multicategories, multifunctors, and multinatural transformations.  Detailed discussion of enriched multicategories is in \cite[Chapters 5 and 6]{cerberusIII} and \cite{yau-operad}.

Our application of interest will be the case $\V = (\Cat, \times, \boldone, \xi)$ of small categories and functors with the Cartesian product $\times$.  
We begin with some notation.

\begin{definition}\label{def:profile}
Suppose $C$\label{notation:s-class} is a class.  
\begin{itemize}
\item Denote by\label{notation:profs}
\[\Prof(C) = \coprodover{n \geq 0}\ C^{\times n}\] 
the class of finite ordered sequences of elements in $C$.  An element in $\Prof(C)$ is called a \emph{$C$-profile}.  
\item A typical $C$-profile of length
  $n=\len\angc$ is denoted by $\angc = (c_1, \ldots, c_n) \in
  C^{n}$\label{notation:us} or by $\ang{c_i}_i$ to indicate
  the indexing variable.  The empty $C$-profile
  is denoted by $\ang{}$.
\item We let $\oplus$ denote the \label{not:concat}concatenation of profiles, and note
  that $\oplus$ is an associative binary operation with unit given by
  the empty tuple $\ang{}$.
\item An element in $\Prof(C)\times C$ is denoted
  as\label{notation:duc} $\IMMduc$ with $c'\in C$ and
  $\angc\in\Prof(C)$.
  \defmark
\end{itemize}
\end{definition}

\begin{convention}[Symmetric Monoidal $\V$]\label{explanation:vst-for-enrichment}
Throughout this section we assume that 
\[\V = (\V,\otimes,\tu,\alpha,\lambda,\rho,\xi)\]
is a symmetric monoidal category.  Throughout the rest of this work, unless otherwise specified, we assume that each iterated monoidal product is \emph{left normalized} with the left half of each pair of parentheses at the far left.  For example, we denote
\[a \otimes b \otimes c \otimes d = \big((a \otimes b) \otimes c\big) \otimes d.\]
With this convention, we omit most of the parentheses for iterated monoidal products and tacitly insert the necessary associativity and unit isomorphisms.  This is valid because there is a strong symmetric monoidal adjoint equivalence $\V \to \V_\st$ with $\V_\st$ a permutative category \cite[1.3.10]{cerberusI}.  Because strictification is an equivalence, the strict diagrams commute if and only if their preimages in $\V$ commute.
\end{convention}

\begin{definition}\label{def:enr-multicategory}
A \emph{$\V$-multicategory} $(\M, \gamma, \operadunit)$\label{notation:enr-multicategory} consists of the following data.
\begin{itemize}
\item $\M$ is equipped with a class $\ObM$
  of \emph{objects}.  We write $\Prof(\M)$
  for $\Prof(\Ob\M)$.
\item For $c'\in\ObM$ and $\angc=(c_1,\ldots,c_n)\in\ProfM$, $\M$ is
  equipped with an object of $\V$\label{notation:enr-cduc}
  \[\M\IMMduc = \M\mmap{c'; c_1,\ldots,c_n} \in \V\]
  called the \emph{$n$-ary operation object} or \emph{multimorphism object} with \emph{input profile} $\angc$ and
  \emph{output} $c'$.
\item
For $\IMMduc \in \ProfMM$ as above and a permutation $\sigma \in
\Sigma_n$, $\M$ is equipped with an isomorphism in $\V$
\[\begin{tikzcd}\M\IMMduc \rar{\sigma}[swap]{\cong} & \M\IMMducsigma,\end{tikzcd}\]
called the \emph{right action} or the \emph{symmetric group action}, in which\label{enr-notation:c-sigma}
\[\angc\sigma = (c_{\sigma(1)}, \ldots, c_{\sigma(n)})\]
is the right permutation of $\angc$ by $\sigma$.
\item For $c \in \ObM$, $\M$ is equipped with a morphism\label{notation:enr-unit-c}
\[\operadunit_c\cn \tensorunit \to \M\IMMcc,\] called the \emph{$c$-colored unit}.
\item For
  $c'' \in \ObM$, $\ang{c'} = (c'_1,\ldots,c'_n) \in \ProfM$, and
  $\ang{c_j} = (c_{j,1},\ldots,c_{j,k_j}) \in \ProfM$ for each $j\in\{1,\ldots,n\}$, let
  $\angc = \oplus_j\ang{c_j} \in \ProfM$
  be the concatenation of the $\ang{c_j}$.  Then
  $\M$ is equipped with a
  morphism in $\V$\label{notation:enr-multicategory-composition}
  \begin{equation}\label{eq:enr-defn-gamma}
    \begin{tikzcd}
      \M\mmap{c'';\ang{c'}} \otimes
      \bigotimes\limits_{j=1}^n \M\mmap{c_j';\ang{c_j}}
      \rar{\gamma}
      &
      \M\mmap{c'';\ang{c}}
    \end{tikzcd}
  \end{equation}
called the \emph{composition} or \emph{multicategorical composition}. 
\end{itemize}
These data are required to satisfy the following axioms.
\begin{description}
\item[Symmetric Group Action]
For $\mmap{c';\ang{c}}\in\ProfMM$ with $n=\len\ang{c}$ and
$\sigma,\tau\in\Sigma_n$, the following diagram in $\V$ commutes. 
\begin{equation}\label{enr-multicategory-symmetry}
\begin{tikzcd}
\M\IMMduc \arrow{rd}[swap]{\sigma\tau} \rar{\sigma} & \M\IMMducsigma \dar{\tau}\\
& \M\mmap{c';\angc\sigma\tau}
\end{tikzcd}
\end{equation}
Moreover, the identity permutation in $\Sigma_n$ acts as the identity morphism of $\M\IMMduc$.
\item[Associativity]
Suppose given
\begin{itemize}
\item $c''' \in \ObM$,
\item $\ang{c''} = (c''_1,\ldots,c''_{n}) \in \ProfM$,
\item $\ang{c_j'} = (c'_{j,1},\ldots,c'_{j,k_{j}}) \in \ProfM$ for each
  $j \in \{1,\ldots,n\}$, and
\item $\ang{c_{j,i}} = (c_{j,i,1},\ldots,c_{j,i,\ell_{j,i}}) \in \ProfM$ for
  each $j\in\{1,\ldots,n\}$ and each $i \in \{1,\ldots,k_j\}$,
\end{itemize}
such that $k_j = \len\ang{c_j'} > 0$ for at least one $j$.  For each $j$,
let $\ang{c_j} = \oplus_{i=1}^{k_j}{\ang{c_{j,i}}}$ denote the concatenation of
the $\ang{c_{j,i}}$.  Let $\ang{c} =
\oplus_{j=1}^n{\ang{c_{j}}}$ denote the concatenation of the $\ang{c_j}$.
Let $\ang{c'} = \oplus_{j=1}^n \ang{c'_j}$ denote the concatenation of the $\ang{c'_j}$.
Then the \emph{associativity diagram} below commutes.
\begin{equation}\label{enr-multicategory-associativity}
\begin{tikzpicture}[x=40mm,y=17mm,vcenter]
  \draw[0cell=.85] 
  (0,0) node (a) {\textstyle
    \M\mmap{c''';\ang{c''}}
    \otimes
    \biggl[\bigotimes\limits_{j=1}^n \M\mmap{c''_j;\ang{c'_{j}}}\biggr]
    \otimes
    \bigotimes\limits_{j=1}^n \biggl[\bigotimes\limits_{i=1}^{k_j} \M\mmap{c'_{j,i};\ang{c_{j,i}}}\biggr] 
  }
  (1,.8) node (b) {\textstyle
    \M\mmap{c''';\ang{c'}}
    \otimes
    \bigotimes\limits_{j=1}^{n} \biggl[\bigotimes\limits_{i=1}^{k_j} \M\mmap{c'_{j,i};\ang{c_{j,i}}}\biggr]
  }
  (0,-1.2) node (a') {\textstyle
    \M\mmap{c''';\ang{c''}} \otimes
    \bigotimes\limits_{j=1}^n \biggl[\M\mmap{c_j'';\ang{c_j'}} \otimes \bigotimes\limits_{i=1}^{k_j} \M\mmap{c'_{j,i};\ang{c_{j,i}}}\biggr]
  }
  (1,-2) node (b') {\textstyle
    \M\mmap{c''';\ang{c''}} \otimes \bigotimes\limits_{j=1}^n \M\mmap{c_j'';\ang{c_{j}}}
  }
  (1.2,-.6) node (c) {\textstyle
    \M\mmap{c''';\ang{c}}
  }
  ;
  \draw[1cell=.85]
  (a) edge node {\iso} node['] {\mathrm{permute}} (a')
  (a) edge[shorten >=-2ex] node[pos=.3] {(\ga,1)} (b)
  (b) edge node {\ga} (c)
  (a') edge['] node {(1,\textstyle\bigotimes_j \ga)} (b')
  (b') edge['] node {\ga} (c)
  ;
\end{tikzpicture}
\end{equation}

\item[Unity]
Suppose $c' \in \ObM$.
\begin{enumerate}
\item If $\angc = (c_1,\ldots,c_n) \in \ProfM$ has length $n \geq 1$,
  then the following \emph{right unity diagram}
  is commutative.  Here $\tensorunit^n$ is
  the $n$-fold monoidal product of $\tensorunit$ with itself.
  \begin{equation}\label{enr-multicategory-right-unity}
    \begin{tikzcd} \M\IMMduc \otimes \tensorunit^{n} \dar[swap]{1 \otimes (\otimes_j \operadunit_{c_j})} \rar{\rho} & \M\IMMduc \dar{1}\\
      \M\IMMduc \otimes \bigotimes\limits_{j=1}^n \M\IMMcjcj \rar{\gamma} & \M\IMMduc
    \end{tikzcd}
  \end{equation}

\item
For any $\ang{c} \in \ProfM$, the \emph{left unity diagram}
below is commutative.
\begin{equation}\label{enr-multicategory-left-unity}
\begin{tikzcd}
\tensorunit \otimes \M\IMMduc \dar[swap]{\operadunit_{c'} \otimes 1} \rar{\lambda} & 
\M\IMMduc \dar{1}\\
\M\mmap{c';c'} \otimes \M\IMMduc \rar{\gamma} & \M\IMMduc
\end{tikzcd}
\end{equation}
\end{enumerate}
\item[Equivariance]
Suppose that in the definition of $\gamma$ \cref{eq:enr-defn-gamma}, $\mathrm{len}\ang{c_j} = k_j \geq 0$.
\begin{enumerate}
\item For each $\sigma \in \Sigma_n$, the following \emph{top equivariance diagram} is commutative.
\begin{equation}\label{enr-operadic-eq-1}
\begin{tikzcd}[column sep=large,cells={nodes={scale=.9}},
every label/.append style={scale=.9}]
\M\mmap{c'';\ang{c'}} \otimes \bigotimes\limits_{j=1}^n \M\mmap{c'_j;\ang{c_j}} 
\dar[swap]{\gamma} \rar{(\sigma, \sigma^{-1})}
& \M\mmap{c'';\ang{c'}\sigma} \otimes \bigotimes\limits_{j=1}^n \M\mmap{c'_{\sigma(j)};\ang{c_{\sigma(j)}}} \dar{\gamma}\\
\M\mmap{c'';\ang{c_1},\ldots,\ang{c_n}} \rar{\sigma\langle k_{\sigma(1)}, \ldots , k_{\sigma(n)}\rangle}
& \M\mmap{c'';\ang{c_{\sigma(1)}},\ldots,\ang{c_{\sigma(n)}}}
\end{tikzcd}
\end{equation}
Here $\sigma\langle k_{\sigma(1)}, \ldots , k_{\sigma(n)} \rangle \in \Sigma_{k_1+\cdots+k_n}$\label{notation:enr-block-permutation} is right action of the block permutation  that permutes the $n$ consecutive blocks of lengths $k_{\sigma(1)}$, $\ldots$, $k_{\sigma(n)}$ as $\sigma$ permutes $\{1,\ldots,n\}$, leaving the relative order within each block unchanged.
\item
Given permutations $\tau_j \in \Sigma_{k_j}$ for $1 \leq j \leq n$,
the following \emph{bottom equivariance
  diagram} is commutative.
\begin{equation}\label{enr-operadic-eq-2}
\begin{tikzcd}[cells={nodes={scale=.9}},
every label/.append style={scale=.9}]
\M\mmap{c'';\ang{c'}} \otimes \bigotimes\limits_{j=1}^n \M\mmap{c'_j;\ang{c_j}}
\dar[swap]{\gamma} \rar{(1, \otimes_j \tau_j)} & 
\M\mmap{c'';\ang{c'}} \otimes \bigotimes\limits_{j=1}^n \M\mmap{c'_j;\ang{c_j}\tau_j}\dar{\gamma} \\
\M\mmap{c'';\ang{c_1},\ldots,\ang{c_n}} \rar{\tau_1 \times \cdots \times \tau_n}
& \M\mmap{c'';\ang{c_1}\tau_1,\ldots,\ang{c_n}\tau_n}
\end{tikzcd}
\end{equation}
Here the block sum $\tau_1 \times\cdots \times\tau_n \in \Sigma_{k_1+\cdots+k_n}$\label{notation:enr-block-sum} is the image of $(\tau_1, \ldots, \tau_n)$ under the canonical inclusion \[\Sigma_{k_1} \times \cdots \times \Sigma_{k_n} \to \Sigma_{k_1 + \cdots + k_n}.\]
\end{enumerate}
\end{description}
This finishes the definition of a $\V$-multicategory.  

Moreover, we define the following.
\begin{itemize}
\item A $\V$-multicategory is \emph{small} if its class of objects is a set.
\item A \emph{$\V$-operad} is a $\V$-multicategory with one object.  If $\M$ is a $\V$-operad, then its object of $n$-ary operations is denoted by $\M_n \in \V$.
\item A \emph{multicategory} is a $\Set$-multicategory, where $(\Set,\times,*)$ is the symmetric monoidal category of sets and functions with the Cartesian product.
\item An \emph{operad} is a $\Set$-operad, that is, a multicategory with one object.
\item A \emph{non-symmetric $\V$-multicategory} is defined as above, except it does not have a designated symmetric group action or satisfy the related action or equivariance axioms.
  \defmark
\end{itemize}
\end{definition}

\begin{definition}\label{definition:IT}
  The \emph{initial multicategory} has an empty set of objects.
  The \emph{initial operad} $\Mtu$ consists of a single object $*$ and a single operation, which is the unit $1_*$.  The \emph{terminal multicategory} $\Mterm$ consists of a single object $*$ and a single $n$-ary operation $\iota_n$ for each $n \ge 0$.
\end{definition}

\begin{example}[Endomorphism Operad]\label{example:enr-End}
  Suppose $\M$ is a $\V$-multicategory and $c$ is an object
  of $\M$.  Then $\End(c)$ is the $\V$-operad consisting of
  the single object $c$ and $n$-ary operation object
  \[\End(c)_n = \M\mmap{c;\ang{c}} \in \V,\]
  where $\ang{c}$ denotes the constant $n$-tuple at $c$.  The
  symmetric group action, unit, and composition of $\End(c)$ are given
  by those of $\M$.
\end{example}

\begin{example}[Underlying Enriched Category]\label{ex:unarycategory}
Each $\V$-multicategory $(\M,\ga,\operadunit)$ has an underlying $\V$-category with
\begin{itemize}
\item the same objects,
\item identities given by the colored units, and
\item composition given by 
\[\begin{tikzcd}[column sep=large]
\M\scmap{b;c} \otimes \M\scmap{a;b} \ar{r}{\gamma} & \M\scmap{a;c}
\end{tikzcd}\]
for objects $a, b, c \in \M$.
\end{itemize}   
The $\V$-category associativity and unity diagrams as in e.g., \cite[1.2.1]{cerberusIII}, are the unary special cases of, respectively, the associativity diagram \cref{enr-multicategory-associativity} and the unity diagrams \cref{enr-multicategory-right-unity,enr-multicategory-left-unity} of a $\V$-multicategory.
\end{example}

\begin{definition}\label{def:enr-multicategory-functor}
A \emph{$\V$-multifunctor} $P \cn \M \to \N$ between $\V$-multicategories $\M$ and $\N$ consists of
\begin{itemize}
\item an object assignment $P \cn \ObM \to \ObN$ and
\item for each $\mmap{c';\ang{c}} \in \ProfMM$ with $\angc=(c_1,\ldots,c_n)$, a component morphism in $\V$
\[P \cn \M\mmap{c';\ang{c}} \to \N\mmap{Pc';P\ang{c}},\] where $P\angc=(Pc_1,\ldots,Pc_n)$.
\end{itemize}
These data are required to preserve the symmetric group action, the colored units, and the composition in the following sense.
\begin{description}
\item[Symmetric Group Action] For each $\IMMduc$ as above and each
  permutation $\sigma \in \Sigma_n$, the following
  diagram in $\V$ is commutative.
\begin{equation}\label{enr-multifunctor-equivariance}
\begin{tikzcd}
\M\mmap{c';\ang{c}} \ar{d}{\cong}[swap]{\sigma} \ar{r}{P} & \N\mmap{Pc';P\ang{c}} \ar{d}{\cong}[swap]{\sigma}\\
\M\mmap{c';\ang{c}\sigma} \ar{r}{P} & \N\mmap{c';P\ang{c}\sigma}\end{tikzcd}
\end{equation}
\item[Units] For each $c\in\ObM$, the following diagram in $\V$ is commutative.
\begin{equation}\label{enr-multifunctor-unit}
\begin{tikzpicture}[x=25mm,y=15mm,vcenter]
  \draw[0cell] 
  (0,0) node (a) {\tu}
  (1,.5) node (b) {\M\mmap{c;c}}
  (1,-.5) node (b') {\N\mmap{Pc;Pc}}
  ;
  \draw[1cell] 
  (a) edge node {\operadunit_c} (b)
  (a) edge node[swap,pos=.6] {\operadunit_{Pc}} (b')
  (b) edge node {P} (b')
  ;
\end{tikzpicture}
\end{equation} 
\item[Composition] For $c''$, $\ang{c'}$, and $\ang{c} =
  \oplus_j\ang{c_j}$ as in the definition of $\gamma$
  \cref{eq:enr-defn-gamma}, the following diagram in $\V$ is commutative.
\begin{equation}\label{v-multifunctor-composition}
\begin{tikzcd}[column sep=large,cells={nodes={scale=.9}},
every label/.append style={scale=.9}]
\M\mmap{c'';\ang{c'}} \otimes \bigotimes\limits_{j=1}^n \M\mmap{c'_j;\ang{c_j}} \dar[swap]{\gamma} \ar{r}{(P,\otimes_j P)} & \N\mmap{Pc'';P\ang{c'}} \otimes \bigotimes\limits_{j=1}^n \N\mmap{Pc'_j;P\ang{c_j}} \dar{\gamma}\\  
\M\mmap{c'';\ang{c}} \ar{r}{P} & \N\mmap{Pc'';P\ang{c}}
\end{tikzcd}
\end{equation}
\end{description}
This finishes the definition of a $\V$-multifunctor.  

Moreover, we define the following.
\begin{enumerate}
\item A \emph{multifunctor} is a $\Set$-multifunctor.
\item A $\V$-multifunctor $\M \to \N$ is also called an \emph{$\M$-algebra in $\N$}.
\item For another $\V$-multifunctor $Q \cn \N\to\sfL$ between $\V$-multicategories, where $\sfL$ has object class $\Ob\sfL$, the \emph{composition} $QP \cn \M\to\sfL$ is the $\V$-multifunctor defined by composing the assignments on objects 
\[\begin{tikzcd} \ObM \ar{r}{P} & \ObN \ar{r}{Q} & \Ob\sfL
\end{tikzcd}\]
and the morphisms on $n$-ary operations
\[\begin{tikzcd}
\M\mmap{c';\ang{c}} \ar{r}{P} & \N\mmap{Pc';P\ang{c}} \ar{r}{Q} & \sfL\mmap{QPc';QP\ang{c}}.
\end{tikzcd}\]
\item The \emph{identity $\V$-multifunctor} $1_{\M} \cn \M\to\M$ is defined by the identity assignment on objects and the identity morphism on $n$-ary operations.
\item A \emph{$\V$-operad morphism} is a $\V$-multifunctor between two $\V$-multicategories with one object.
\item\label{it:non-symm-multifunctor} A \emph{non-symmetric} $\V$-multifunctor consists of the same data as above, but is not required to preserve the symmetric group action of its source and target.  Thus a non-symmetric $\V$-multifunctor is only required to preserve the colored units and composition. \defmark
\end{enumerate}
\end{definition}

\begin{definition}\label{def:enr-multicat-natural-transformation}
Suppose $P,Q \cn \M\to\N$ are $\V$-multifunctors as in \cref{def:enr-multicategory-functor}.  A \emph{$\V$-multinatural transformation} $\theta \cn P\to Q$ consists of component morphisms in $\V$
\[\theta_c \cn \tu \to \N\mmap{Qc;Pc} \forspace c\in\ObM\] 
such that the following \emph{$\V$-naturality diagram} in $\V$
commutes for each $\mmap{c';\ang{c}} \in \ProfMM$ with
$\angc=(c_1,\ldots,c_n)$.
\begin{equation}\label{enr-multinat}
\begin{tikzpicture}[x=25mm,y=12mm,vcenter]
  \draw[0cell=.85]
  (.5,0) node (a) {\M\mmap{c';\ang{c}}}
  (1,1) node (b) {\tu \otimes \M\mmap{c';\ang{c}}}
  (3,1) node (c) {\N\mmap{Qc';Pc'} \otimes \N\mmap{Pc';P\ang{c}}}
  (3.5,0) node (d) {\N\mmap{Qc';P\ang{c}}}
  (1,-1) node (b') {\M\mmap{c';\ang{c}}\otimes \bigotimes\limits_{j=1}^n \tu}
  (3,-1) node (c') {\N\mmap{Qc';Q\ang{c}} \otimes \bigotimes\limits_{j=1}^n \N\mmap{Qc_j;Pc_j}}
  ;
  \draw[1cell=.85] 
  (a) edge node[pos=.3] {\la^\inv} (b)
  (a) edge node[swap,pos=.2] {\rho^\inv} (b')
  (b) edge node {\theta_{c'} \otimes P} (c)
  (c) edge node[pos=.7] {\ga} (d)
  (b') edge node {Q \otimes \bigotimes\limits_{j=1}^n \theta_{c_j}} (c')
  (c') edge['] node[pos=.7] {\ga} (d)
  ;
\end{tikzpicture}
\end{equation}
This finishes the definition of a $\V$-multinatural transformation.  

Moreover, we define the following.
\begin{itemize}
\item The \emph{identity $\V$-multinatural transformation} $1_P \cn P\to P$ has components \[(1_P)_c = \operadunit_{Pc} \forspace c\in\ObM.\]
\item A \emph{multinatural transformation} is a $\Set$-multinatural transformation.
\item For non-symmetric multifunctors $P,Q\cn \M \to \N$, a $\V$-multinatural transformation $\theta\cn P \to Q$ has the same definition given above, and we use the same terminology.
\defmark
\end{itemize} 
\end{definition}

\begin{definition}\label{def:enr-multinatural-composition}
Suppose $\theta \cn P \to Q$ is a $\V$-multinatural transformation between $\V$-multifunctors as in \cref{def:enr-multicat-natural-transformation}.
\begin{enumerate}
\item Suppose $\beta \cn Q \to R$ is a $\V$-multinatural transformation for
  a $\V$-multifunctor $R \cn \M \to \N$.  The \emph{vertical
    composition}\label{notation:enr-operad-vcomp}
\begin{equation}\label{multinatvcomp}
  \beta\theta \cn P \to R
\end{equation} 
is the $\V$-multinatural transformation with components at $c \in \ObM$ given by the following composites in $\V$.
  \[
  \begin{tikzpicture}[x=45mm,y=15mm]
    \draw[0cell] 
    (0,0) node (a) {\tu}
    (0,-1) node (b) {\tu \otimes \tu}
    (1,-1) node (c) {\N\mmap{Rc;Qc} \otimes \N\mmap{Qc;Pc}}
    (1,0) node (d) {\N\mmap{Rc;Pc}}
    ;
    \draw[1cell] 
    (a) edge node['] {\la^\inv} (b)
    (b) edge node {\be_c \otimes \theta_c} (c)
    (c) edge['] node {\ga} (d)
    (a) edge node {(\be\theta)_c} (d)
    ;
  \end{tikzpicture}
  \]
\item Suppose $\theta' \cn P' \to Q'$ is a $\V$-multinatural transformation for $\V$-multifunctors $P', Q' \cn \N \to \sfL$.  The \emph{horizontal composition}\label{notation:enr-operad-hcomp}
\begin{equation}\label{multinathcomp}
\theta' \ast \theta \cn P'P \to Q'Q
\end{equation} 
is the $\V$-multinatural
transformation with components at $c \in \ObM$ given by the following composites in $\V$.
\[
\begin{tikzpicture}[x=50mm,y=12mm]
  \draw[0cell=.9] 
  (0,0) node (a) {\tu}
  (0,-2) node (b) {\tu \otimes \tu}
  (1,-2) node (c) {\sfL\mmap{Q'Qc;P'Qc} \otimes \N\mmap{Qc;Pc}}
  (1,-1) node (d) {\sfL\mmap{Q'Qc;P'Qc} \otimes \sfL\mmap{P'Qc;P'Pc}}
  (1,0) node (e) {\sfL\mmap{Q'Qc;P'Pc}}
  ;
  \draw[1cell=.9] 
  (a) edge node {(\theta' * \theta)_c} (e)
  (a) edge['] node {\la^\inv} (b)
  (b) edge node {\theta'_{Qc} \otimes \theta_c} (c)
  (c) edge['] node {1 \otimes P'} (d)
  (d) edge['] node {\ga} (e)
  ;
\end{tikzpicture}
\]
\end{enumerate}
This finishes the definition.
\end{definition}

\section{The \texorpdfstring{$\Cat$}{Cat}-Multicategory of Small Multicategories}\label{sec:Multicat}

The 2-category $\Multicat$ of small multicategories, multifunctors, and multinatural transformations has a closed symmetric monoidal structure given by the Boardman-Vogt tensor product.  This induces a $\Cat$-multicategory structure that we will describe below.  We give only those details necessary for the arguments in the sequel, and refer the reader to \cite[Chapters~5 and~6]{cerberusIII} for a full treatment of related background.

We begin with preliminary notation.
\begin{definition}\label{definition:xiotimes}
  Given profiles $\ang{c} \in \Prof(C)$ and $\ang{d} \in \Prof(D)$
  with $m = \len\ang{c}$ and $n = \len\ang{d}$, we define the following profiles.
  \begin{align*}
    \ang{c} \times d_j & = \ang{(c_i,d_j)}_{i=1}^m = \big((c_1,d_j),(c_2,d_j),\ldots, (c_m,d_j)\big)\\
    c_i \times \ang{d} & = \ang{(c_i,d_j)}_{j=1}^n = \big((c_i,d_1),(c_i,d_2),\ldots, (c_i,d_n)\big)\\
    \ang{c} \otimes \ang{d} & = \ang{\ang{(c_i,d_j)}_{i=1}^m}_{j=1}^n
    = \big(\angc \times d_1, \ldots , \angc \times d_n\big)\\
    \ang{c} \otimes^\transp \ang{d} & = \ang{\ang{(c_i,d_j)}_{j=1}^n}_{i=1}^m
    = \big(c_1 \times \angd, \ldots , c_m \times \angd\big)
  \end{align*}
Thus, $\otimes$ uses the reverse lexicographic order, and $\otimes^\transp$ uses the lexicographic order.  We denote by 
\[\xitimes = \xitimes_{m,n} \cn \ang{c} \otimes \ang{d} \fto{\iso} \ang{c} \otimes^\transp \ang{d}\]
the permutation induced by changing order of indexing.
\end{definition}

\begin{example}\label{ex:profile-tensor}
Suppose $\ang{c} = \{c_1,c_2\}$ and $\ang{d} = \{d_1,d_2,d_3\}$.  Then the profiles $\angc \otimes \angd$ and $\angc \otimes^\transp \angd$ are given as follows.
\[\begin{split}
\angc \otimes \angd &= \big\{(c_1,d_1), (c_2,d_1), (c_1,d_2), (c_2,d_2), (c_1,d_3), (c_2,d_3) \big\}\\
\angc \otimes^\transp \angd &= \big\{(c_1,d_1), (c_1,d_2), (c_1,d_3), (c_2,d_1), (c_2,d_2), (c_2,d_3)\big\}
\end{split}\]
Note that $\xitimes_{m,1}$ and $\xitimes_{1,n}$ are identity permutations.
\end{example}

\begin{remark}[Choice of Ordering]\label{rk:reverse-lex}
Throughout this paper we consistently use the reverse lexicographic ordering $\otimes$.  This is simply a matter of convenience.  In other words, we can also consistently use the lexicographic ordering $\otimes^\transp$ throughout, and our main results remain valid, as we explain further in \cref{remark:lex-order}.  The reason we prefer $\otimes$ over $\otimes^\transp$ is that, in the profile $\angc \otimes \angd$, the indices $i$ and $j$ appear in the same left-to-right order as they do in $(c_i,d_j)$.  Some of our main constructions involve iterating the tensor product; see \cref{eq:xonen,eq:phionen}.  If we use $\otimes^\transp$ instead of $\otimes$, then \cref{eq:xonen} would involve the indices 
\[1 \leq j_n \leq r_n, \quad 1 \leq j_{n-1} \leq r_{n-1}, \quad \ldots,\quad 1 \leq j_1 \leq r_1,\]
further complicating the formulas.
\end{remark}

\begin{definition}[Boardman-Vogt Tensor Product of Multicategories]\label{definition:BVtensor}
Suppose given small multicategories $\M$ and $\N$.  The \emph{tensor product} $\M \otimes \N$ is the multicategory defined as follows.  Its object set is $\Ob\M \times \Ob\N$.  For $c \in \Ob\M$ and $d \in \Ob\N$, the corresponding object in $\M \otimes \N$ is denoted $(c,d)$ or $c \otimes d$.

The operations in $\M \otimes \N$ are generated by the following operations for $c \in \Ob\M$, $d \in \Ob\N$, $\phi \in \M\scmap{\ang{c_i}_{i=1}^m;c'}$, and $\psi \in \N\scmap{\ang{d_j}_{j=1}^n;d'}$.
\[\begin{split}
\phi \otimes d & \in (\M \otimes \N)\mmap{(c',d);\ang{(c_i,d)}_{i=1}^m}\\
c \otimes \psi & \in (\M \otimes \N)\mmap{(c,d');\ang{(c,d_j)}_{j=1}^n}
\end{split}\]
These data are subject to relations \cref{it:BV-1,it:BV-2,it:BV-3,it:BV-4,it:BV-5,it:interchange-relation} below.
  \begin{enumerate}
  \item\label{it:BV-1} For $c \in \M$ and $d \in \N$, we have
    \[
      1_c \otimes d = 1_{(c,d)} = c \otimes 1_d.
    \]
  \item\label{it:BV-2} For operations $\phi, \phi_1, \ldots, \phi_n$  in $\M$ such that the
    composite below is defined, we have
    \[
      \ga\scmap{\phi \otimes d ; \phi_1 \otimes d, \ldots, \phi_n \otimes d} =
      \ga\scmap{\phi ; (\phi_1,\ldots, \phi_n)} \otimes d.
    \]
  \item\label{it:BV-3} For $\si \in \Si_m$, we have
    \[
      (\phi \otimes d)\cdot\si = (\phi\cdot\si) \otimes d.
    \]
  \item\label{it:BV-4} For operations $\psi, \psi_1, \ldots, \psi_m$ in $\N$ such that the
    composite below is defined, we have
    \[
      \ga\scmap{c \otimes \psi ; c \otimes \psi_1, \ldots, c \otimes \psi_m} =
      c \otimes \ga\scmap{\psi ; \psi_1, \ldots, \psi_m}.
    \]
  \item\label{it:BV-5} For $\si \in \Si_n$, we have
    \[
      (c \otimes \psi)\cdot\si = c \otimes (\psi\cdot\si).
    \]
  \item\label{it:interchange-relation} For operations $\phi \in \M\scmap{\ang{c_i}_{i=1}^m;c'}$ and $\psi \in \N\scmap{\ang{d_j}_{j=1}^n;d'}$, we have
\[\ga\scmap{c' \otimes \psi; \phi \otimes d_1, \ldots, \phi \otimes d_n} =
\ga\scmap{\phi \otimes d'; c_1 \otimes \psi, \ldots, c_m \otimes \psi} \cdot\xitimes\]  
in $(\M \otimes \N)\scmap{\angc \otimes \angd; (c',d')}$.  This is called the \emph{interchange relation}.  
\end{enumerate}  
This finishes the definition of $\M \otimes \N$.  Moreover, we define the operations
\begin{equation}\label{phi-tensor-psi}
\begin{split}
\phi \otimes \psi &= \ga\scmap{c' \otimes \psi; \phi \otimes d_1, \ldots, \phi \otimes d_n} \andspace\\
\phi \otimes^\transp \psi &= \ga\scmap{\phi \otimes d'; c_1 \otimes \psi, \ldots, c_m \otimes \psi},
\end{split}
\end{equation}
which are the two composite operations in relation \cref{it:interchange-relation} above.
\end{definition} 

\begin{explanation}\label{explanation:bvtensor}
  If we draw an operation as an arrow from its input profile to its output object, the interchange relation means that the two composites
  \[
  \begin{tikzpicture}[x=30mm,y=30mm]
    \draw[0cell] 
    (0,0) node (a) {\ang{c}\otimes\ang{d}}
    (1,0) node (b) {c'\times\ang{d}}
    (2,-.25) node (z) {(c',d')}
    (0,-.5) node (a') {\ang{c}\otimes^\transp\ang{d}}
    (1,-.5) node (b') {\ang{c}\times d'}
    ;
    \draw[1cell] 
    (a) edge node {\ang{\phi\otimes d_j}_{j=1}^n} (b)
    (a') edge['] node {\ang{c_i \otimes\psi}_{i=1}^m} (b')
    (b) edge node {c'\otimes\psi} (z)
    (b') edge['] node {\phi\otimes d'} (z)
    ; 
  \end{tikzpicture}
  \]
  correspond under the bijection
  \[
  \xitimes\cn
  (\M\otimes\N)\mmap{(c',d');\ang{c}\otimes^\transp\ang{d}} \fto{\iso} 
  (\M\otimes\N)\mmap{(c',d');\ang{c}\otimes\ang{d}}.
  \]
  A multifunctor
  \[
  H:\M \otimes \N \to \sfL
  \]
  consists of an assignment on objects $H(c,d) \in \Ob\sfL$ for
  $(c,d) \in \Ob\M \times \Ob\N$ such that each
  $H(c,-)$ and $H(-,d)$ is a multifunctor
  and such that we have
  \begin{equation}\label{eq:H-interchange-relation}
  H(\phi \otimes \psi) = H(\phi \otimes^\transp \psi)\cdot\xitimes
  \end{equation}
  for each $\phi \in \M\mmap{c';\ang{c}}$ and $\psi \in
  \N\mmap{d';\ang{d}}$.
\end{explanation}

\begin{theorem}[{\cite[5.7.14,~6.4.3]{cerberusIII}}]\label{theorem:Multicat-sm}
  The tensor product of \cref{definition:BVtensor} is a $\Cat$-enriched symmetric monoidal product for $\Multicat$.  Its monoidal unit is the initial operad $\I$ of \cref{definition:IT}.
\end{theorem}

\begin{explanation}\label{expl:multicat-cat-monoidal}
The $\Cat$-enrichment of the symmetric monoidal product for $\Multicat$ is defined as follows.  Given multifunctors
\[F \cn \M \to \M' \andspace G \cn \N \to \N',\]
we define a multifunctor
\[F \otimes G \cn \M \otimes \N \to \M' \otimes \N'\]
with assignment on generating operations given by
\[(F \otimes G)(\phi \otimes d) = F\phi \otimes Gd \andspace
(F \otimes G)(c \otimes \psi) = Fc \otimes G\psi.\]
Moreover, with $F$ and $G$ as above, for multinatural transformations
\[\theta \cn F \to F' \andspace \omega \cn G \to G',\]
we define the tensor product multinatural transformation
\[\theta \otimes \omega \cn F \otimes G \to F' \otimes G'\]
with components
\[(\theta \otimes \omega)_{(c,d)} = \theta_c \otimes \omega_d\]
for $(c,d) \in \Ob\M \times \Ob\N$.  The right-hand side of the above equality is defined as in \cref{phi-tensor-psi}.
\end{explanation}

The closed symmetric monoidal structure on $\Multicat$ induces a $\Cat$-enriched multicategory structure that we now describe.

\begin{definition}\label{definition:Multicat-enr}
  Let $\Multicat$ also denote the $\Cat$-enriched multicategory whose objects are small multicategories and whose category of $n$-ary operations, 
  \[
    \Multicat\scmap{\ang{\M};\N},
  \]
  consists of multifunctors
  \begin{equation}\label{eq:naryop}
    H,K \cn \bigotimes_{i=1}^n \M_i \to \N
  \end{equation} 
  and multinatural transformations $\theta\cn H \to K$.
  
  For a permutation $\si \in \Si_n$, the right action
  \begin{align*}
    \Multicat\scmap{\ang{\M};\N}
    & \fto[\iso]{\si} \Multicat\scmap{\ang{\M}\si;\N}\\
    H & \mapsto \xi_\si \circ H\\
    \theta & \mapsto \xi_\si * \theta
  \end{align*}
  is given by composition and whiskering with the permutation of tensor factors
  \[
    \bigotimes_{i=1}^n \M_{\si(i)} \fto[\iso]{\xi_\si} \bigotimes_{i=1}^n \M_i.
  \]
  
  Units are given by identity multifunctors.  The composition $\ga\scmap{H';\ang{H}}$ for multifunctors
  \begin{align*}
    H_i \cn & \bigotimes_{j = 1}^{n_i} \M_{i,j} \to \M'_i
              \forspace 1 \leq i \leq j, \andspace\\
    H' \cn & \bigotimes_{i = 1}^{n} \M'_{i} \to \M'',
  \end{align*} 
  is given by composition with the associativity isomorphism of $\otimes$
  \begin{equation}\label{eq:Multicat-comp}
    \bigotimes_{i,j}\M_{i,j} \fto[\iso]{\al}
    \bigotimes_{i=1}^n \bigotimes_{j = 1}^{n_i} \M_{i,j} 
    \fto{\otimes_i H_i}
    \bigotimes_{i = 1}^{n} \M'_{i} \fto{H'} \M''
  \end{equation} 
  where the first tensor product is the left normalized tensor product (\cref{explanation:vst-for-enrichment}) of the concatenation of the $\ang{\M_{i}} = \ang{\M_{i,j}}_{j=1}^{n_i}$.  The $\V$-multicategory axioms for $\V = \Cat$ follow from the enriched symmetric monoidal axioms for $\Multicat$.  See \cite[Section~6.3]{cerberusIII} for further details.
\end{definition}

\subsection*{Pointed Multicategories}

Recall from \cref{definition:IT} the terminal multicategory $\Mterm$.
\begin{definition}\label{definition:pointed-multicat}
  A \emph{pointed} multicategory is a multicategory $\M$ together with a distinguished multifunctor
  \[
  \iota \cn \Mterm \to \M
  \]
  called the \emph{basepoint} of $\M$.  Pointed multifunctors and multinatural transformations are those that commute with the basepoint morphisms.  The 2-category of small pointed multicategories, $\Multicat_*$, consists of small pointed multicategories, pointed multifunctors, and pointed multinatural transformations.
\end{definition}

The tensor product of small multicategories induces a smash product of small pointed multicategories.  For small pointed multicategories $\M$ and $\N$, the smash product $\M \sma \N$ is defined as the following pushout in $\Multicat$.
\begin{equation}\label{MsmashN}
\begin{tikzcd}
(\M \otimes \Mterm) \bincoprod (\Mterm \otimes \N) \ar{d} \ar{r} & \M \otimes \N \ar{d}{\varpi_{\M,\N}}\\
\Mterm \ar{r} & \M \sma \N
\end{tikzcd}
\end{equation}
In the diagram above, the left vertical arrow is the unique multifunctor to $\Mterm$.  The top horizontal arrow is induced by the basepoints of $\M$ and $\N$.  The basepoint of $\M \sma \N$ is the bottom horizontal arrow.  Thus $\M \sma \N$ has the same objects as $\M \otimes \N$, and the operations in $\M \sma \N$ are represented by those in $\M \otimes \N$, subject to basepoint conditions.

The smash product provides the $\Cat$-enriched multicategory structure for small permutative categories, which we will describe further in \cref{sec:permcat}.  Most details of the smash product for pointed multicategories will not be needed here, so we refer the reader to \cite[Chapters~4 and~5]{cerberusIII} for a complete treatment.  For our purposes we need only the following result.

\begin{theorem}[{\cite[6.4.4]{cerberusIII}}]\label{theorem:Multicat-ptd}
  The smash product of pointed multicategories is a $\Cat$-enriched symmetric monoidal product for $\Multicat_*$.  Its monoidal unit is the coproduct
  \[
  \Mtu_+ = \Mtu \bincoprod \Mterm.
  \]
\end{theorem}
As with $\Multicat$, the smash product of pointed multicategories induces a $\Cat$-enriched multicategory of small pointed multicategories.
\begin{definition}\label{definition:Multicat-ptd}
  Let $\Multicat_*$ also denote the $\Cat$-enriched multicategory whose objects are small pointed multicategories and whose category of $n$-ary operations
  \[
    \Multicat_*\scmap{\ang{\M};\N} = \Multicat_*\big( \bigwedge_i \M_i, \N \big)
  \]
  consists of pointed multifunctors and pointed multinatural transformations out of an iterated smash product.  The symmetric group action, units, and composition are given by the $\Cat$-enriched symmetric monoidal structure of $(\Multicat_*,\sma, \Mtu_+)$.
\end{definition}

\section{The \texorpdfstring{$\Cat$}{Cat}-Multicategory of Permutative Categories}\label{sec:permcat}

In this section we define the $\Cat$-enriched multicategory of permutative categories, multilinear functors, and multinatural transformations.
For further discussion of plain, braided, and symmetric monoidal categories, we refer the reader to \cite{joyal-street,maclane,johnson-yau,cerberusI,cerberusII}.

\begin{definition}\label{definition:permutativecategory}
  A \emph{permutative category} $(\C,\oplus,e,\xi)$ consists of
  \begin{itemize}
  \item a category $\C$,
  \item a functor $\oplus \cn \C \times \C \to \C$, called the \emph{monoidal sum},
  \item an object $e \in \C$, called the \emph{monoidal unit}, and
  \item a natural isomorphism $\xi$ called the \emph{symmetry isomorphism} with components
    \[
    \xi_{X,Y} \cn X \oplus Y \to Y \oplus X
    \]
    for objects $X$ and $Y$ of $\C$.
  \end{itemize}
  The monoidal sum is required to be associative and unital, with $e$ as its unit.
  The symmetry isomorphism $\xi$ is required to make the following symmetry and hexagon diagrams commute for objects $X,Y,Z \in \C$. 
  \begin{equation}\label{symmoncatsymhexagon}
    \begin{tikzpicture}[xscale=3,yscale=1.2,vcenter]
      \tikzset{0cell/.append style={nodes={scale=.85}}}
      \tikzset{1cell/.append style={nodes={scale=.85}}}
      \def\h{.3}
      \draw[0cell] 
      (0,0) node (a) {X \oplus Y}
      (a)++(.8,0) node (c) {X \oplus Y}
      (a)++(.4,-1) node (b) {Y \oplus X}
      ;
      \draw[1cell] 
      (a) edge node {1_{X \oplus Y}} (c)
      (a) edge node [swap,pos=.3] {\xi_{X,Y}} (b)
      (b) edge node [swap,pos=.7] {\xi_{Y,X}} (c)
      ;
      \begin{scope}[shift={(2,.5)}]
        \draw[0cell] 
        (0,0) node (x11) {(Y \oplus X) \oplus Z}
        (x11)++(.7,0) node (x12) {Y \oplus (X \oplus Z)}
        (x11)++(-110:1.3) node (x21) {(X \oplus Y) \oplus Z}
        (x12)++(-70:1.3) node (x22) {Y \oplus (Z \oplus X)}
        (x21)++(-70:.7) node (x31) {X \oplus (Y \oplus Z)}
        (x22)++(-110:.7) node (x32) {(Y \oplus Z) \oplus X}
        ;
        \draw[1cell]
        (x21) edge node[pos=.25] {\xi_{X,Y} \oplus 1_Z} (x11)
        (x11) edge[equal] node {} (x12)
        (x12) edge node[pos=.75] {1_Y \oplus \xi_{X,Z}} (x22)
        (x21) edge[equal] node[swap,pos=.25] {} (x31)
        (x31) edge node {\xi_{X,Y \oplus Z}} (x32)
        (x32) edge[equal] node[swap,pos=.75] {} (x22)
        ;
      \end{scope}
    \end{tikzpicture}
  \end{equation}
  A permutative category is also called a \emph{strict symmetric monoidal category}.
  The strictness refers to the conditions that the monoidal sum be strictly associative and unital.
\end{definition}

\begin{definition}\label{definition:strictmonoidalfunctor}
  Suppose $\C$ and $\D$ are permutative categories.  A \emph{symmetric monoidal functor}
  \[
    (P,P^2,P^0) \cn \C \to \D
  \]
  consists of a functor $P \cn \C \to \D$ together with natural transformations
  \[
    PX \oplus PY \fto{P^2} P(X \oplus Y) \andspace
    e \fto{P^0} Pe
  \]
  for objects $X,Y \in \C$, called the \emph{monoidal constraint} and \emph{unit constraint}, respectively.  These data satisfy the following associativity, unity, and symmetry axioms.
  \begin{description}
  \item[Associativity] The following diagram is commutative for all objects $X,Y,Z \in \C$.
    \begin{equation}\label{eq:sm-monoidal}
      \begin{tikzpicture}[xscale=4.2,yscale=1.5,vcenter]
        \draw[0cell=.85] 
        (.2,0) node (a0) {
          \big(PX \oplus PY\big) \oplus PZ
        }
        (a0)++(3.8ex,0) node (b0) {
          PX \oplus \big(PY \oplus PZ\big)
        }
        (0,-1) node (a1) {
          P(X \oplus Y) \oplus PZ
        }
        (b0)+(.2,-1) node (b1) {
          PX \oplus P(Y \oplus Z)
        }
        (.2,-2) node (a2) {
          P\big((X \oplus Y) \oplus Z\big)
        }
        (b0)+(0,-2) node (b2) {
          P\big(X \oplus (Y \oplus Z)\big)
        }
        ;
        \draw[1cell=.85] 
        (a0) edge[equal] node {} (b0)
        (a2) edge[equal] node {} (b2)
        (a0) edge['] node {P^2 \oplus 1_{PZ}} (a1)
        (a1) edge['] node[pos=.3] {P^2} (a2)
        (b0) edge node {1_{PX} \oplus P^2} (b1)
        (b1) edge node[pos=.3] {P^2} (b2)
        ;
      \end{tikzpicture}
    \end{equation}

  \item[Unity] The following two diagrams are commutative for all objects $X \in \C$.
    \begin{equation}\label{eq:sm-unit}
      \begin{tikzpicture}[xscale=2.25,yscale=1.5,vcenter]
        \tikzset{0cell/.append style={nodes={scale=.85}}}
        \tikzset{1cell/.append style={nodes={scale=.85}}}
        \draw[0cell] 
        (.5,-.3) node (a) {e \oplus PX}
        (a)++(.5,0) node (b) {PX}
        (0,-1) node (c) {Pe \oplus PX}
        (1,-1) node (d) {P(e \oplus X)}
        ;
        \draw[1cell] 
        (a) edge[equal] node {} (b)
        (c) edge node {P^2} (d)
        (a) edge['] node {P^0 \oplus 1_{PX}} (c)
        (b) edge[equal] node {} (d)
        ;
      \end{tikzpicture}
      \andspace
      \begin{tikzpicture}[xscale=2.25,yscale=1.5,vcenter]
        \tikzset{0cell/.append style={nodes={scale=.85}}}
        \tikzset{1cell/.append style={nodes={scale=.85}}}
        \draw[0cell] 
        (.5,-.3) node (a) {PX \oplus e}
        (a)++(.5,0) node (b) {PX}
        (0,-1) node (c) {PX \oplus Pe}
        (1,-1) node (d) {P(X \oplus e)}
        ;
        \draw[1cell] 
        (a) edge[equal] node {} (b)
        (c) edge node {P^2} (d)
        (a) edge['] node {1_{PX} \oplus P^0} (c)
        (b) edge[equal] node {} (d)
        ;
      \end{tikzpicture}
    \end{equation}
  \item[Symmetry] The following diagram is commutative for all objects $X,Y \in \C$.
    \begin{equation}\label{eq:sm-symm}
      \begin{tikzpicture}[xscale=4,yscale=1.1,vcenter]
        \tikzset{0cell/.append style={nodes={scale=.85}}}
        \tikzset{1cell/.append style={nodes={scale=.85}}}
        \draw[0cell] 
        (0,0) node (a0) {PX \oplus PY}
        (1,0) node (b0) {PY \oplus PX}
        (0,-1) node (a1) {P(X \oplus Y)}
        (1,-1) node (b1) {P(Y \oplus X)}
        ;
        \draw[1cell] 
        (a0) edge node {\xi_{PX,PY}} (b0)
        (a1) edge node {P\xi_{X,Y}} (b1)
        (a0) edge['] node {P^2} (a1)
        (b0) edge node {P^2} (b1)
        ;
      \end{tikzpicture}
    \end{equation}
  \end{description}
  This finishes the definition of a symmetric monoidal functor.

  Composition of symmetric monoidal functors 
  \[
    \C \fto{P} \D \fto{Q} \E
  \]
  is given by composition of underlying functors with monoidal and unit constraints given, respectively, by
  \[
    (QP)^2 = \big( Q(P^2) \big) \circ Q^2
    \andspace
    (QP)^0 = \big( Q(P^0) \big) \circ Q^0.
  \]
  An identity functor is symmetric monoidal with identity monoidal and unit constraints.  A symmetric monoidal functor $P$ is called
  \begin{itemize}
  \item \emph{strong} if $P^0$ and $P^2$ are natural isomorphisms,
  \item \emph{strictly unital} if $P^0$ is the identity natural transformation, and
  \item \emph{strict} if both $P^0$ and $P^2$ are identities. \dqed
  \end{itemize}
\end{definition}

\begin{definition}\label{definition:monoidal-nt}
  Suppose $P,Q\cn \C \to \D$ are symmetric monoidal functors between permutative categories.
  A \emph{monoidal natural transformation}
  \[
  \theta\cn P \to Q
  \]
  is a natural transformation between the underlying functors such that the
  following unity and constraint compatibility diagrams commute for all $X,Y \in \C$.
  \begin{equation}\label{eq:monoidal-nt-constr}
    \begin{tikzpicture}[xscale=2,yscale=1.5,vcenter]
      \tikzset{0cell/.append style={nodes={scale=.85}}}
      \tikzset{1cell/.append style={nodes={scale=.85}}}
      \draw[0cell]
      (0,0) node (e) {e}
      (30:1) node (pe) {Pe}
      (-30:1) node (qe) {Qe}
      ;
      \draw[1cell]
      (e) edge node {P^0} (pe)
      (e) edge['] node {Q^0} (qe)
      (pe) edge node {\theta_e} (qe)
      ;
    \end{tikzpicture}
    \qquad
    \begin{tikzpicture}[xscale=4,yscale=1.5,vcenter]
      \tikzset{0cell/.append style={nodes={scale=.85}}}
      \tikzset{1cell/.append style={nodes={scale=.85}}}
      \draw[0cell] 
      (0,0) node (a) {PX \oplus PY}
      (1,0) node (b) {QX \oplus QY}
      (0,-1) node (a') {P(X \oplus Y)}
      (1,-1) node (b') {Q(X \oplus Y)}
      ;
      \draw[1cell] 
      (a) edge node {\theta_X \oplus \theta_Y} (b)
      (a') edge node {\theta_{X \oplus Y}} (b')
      (a) edge['] node {P^2} (a')
      (b) edge node {Q^2} (b')
      ;
    \end{tikzpicture}
  \end{equation}
  This finishes the definition of a monoidal natural transformation.  Identity and composites of monoidal natural transformations are given by those of the underlying natural transformations.
\end{definition}

\begin{definition}\label{definition:permcat}
  We let $\permcat$ denote the 2-category of small permutative categories, symmetric monoidal functors, and monoidal natural transformations.  We define the following sub-2-categories consisting of the same objects and restricted collections of functors.
  \begin{itemize}
  \item The 1-cells of $\permcatsu$ are strictly unital symmetric monoidal functors.
  \item The 1-cells of $\permcatsus$ are strictly unital strong symmetric monoidal functors.
  \item The 1-cells of $\permcatst$ are strict symmetric monoidal functors.
  \end{itemize}
  In each case the 2-cells consist of all monoidal natural transformations between the indicated 1-cells.
\end{definition}

\subsection*{Multilinear Functors and Transformations}

Now we recall the definitions of multilinear functors and transformations between them.
See \cite[Definition~3.2]{elmendorf-mandell} and \cite[Sections~6.5 and~6.6]{cerberusIII} for further details and discussion of these structures.
Throughout this section, suppose $\C_1, \ldots, \C_n$, and $\D$ are permutative categories.  

\begin{notation}\label{notation:compk}
Suppose 
\[\ang{x} = (x_1, \ldots, x_n)\]
is an $n$-tuple of symbols, and $x_k'$ is a symbol for $k \in \{1,\ldots,n\}$.  We denote by\label{not:compk}
\[\ang{x \compk x_k'} = \ang{x} \compk x_k' 
= \big(\underbracket[0.5pt]{x_1, \ldots, x_{k-1}}_{\text{empty if $k=1$}}, x_k', \underbracket[0.5pt]{x_{k+1}, \ldots, x_n}_{\text{empty if $k=n$}}\big)\]
the $n$-tuple obtained from $\ang{x}$ by replacing its $k$-th entry by $x_k'$.  Similarly, for $k \not= \ell \in \{1,\ldots,n\}$ and a symbol $x'_\ell$, we denote by
\[\ang{x \compk x_k' \comp_\ell x'_\ell} = \ang{x} \compk x_k' \comp_\ell x'_\ell\]
the $n$-tuple obtained from $\ang{x \compk x_k'}$ by replacing its $\ell$-th entry by $x'_\ell$.  We sometimes use the notation
\[\ang{x \compk x_k} = \angx\]
to emphasize the $k$-th term, $x_k$, in $\angx$.  See, for example, the first term in \cref{laxlinearityconstraints}.
\end{notation}

\begin{definition}\label{def:nlinearfunctor}
An \emph{$n$-linear functor}
\[\begin{tikzcd}[column sep=2.3cm]
\prod\limits_{j=1}^n \C_j \ar{r}{\big(P,\, \{P^2_j\}_{j=1}^n\big)} & \D
\end{tikzcd}\]
consists of
\begin{itemize}
\item a functor $P \cn \C_1 \times \cdots \times \C_n \to \D$ and
\item for each $j\in \{1,\ldots,n\}$, a natural transformation $P^2_j$, called the \emph{$j$-th linearity constraint}, with component morphisms
\begin{equation}\label{laxlinearityconstraints}
\begin{tikzcd}[column sep=large]
P\ang{X\compj X_j} \oplus P\ang{X\compj X_j'} \ar{r}{P^2_j}
& P\ang{X\compj (X_j \oplus X_j')} \in \D
\end{tikzcd}
\end{equation}
for objects $\ang{X} \in \txprod_j \C_j$ and $X_j' \in \C_j$.
\end{itemize}
These data are subject to the following five axioms.
\begin{description}
\item[Unity]
  For objects $\angX$ and morphisms $\angf$ in $\txprod_j \C_j$, the following object and morphism unity axioms hold for each $j \in \{1,\ldots,n\}$.
  \begin{equation}\label{nlinearunity}
    \left\{\begin{aligned}
        P \ang{X \compj e} &= e\\
        P \ang{f \compj 1_e} &= 1_e
      \end{aligned}\right.
  \end{equation}
\item[Constraint Unity]
  \begin{equation}\label{constraintunity}
    P^2_j = 1 \qquad \text{if any $X_i = e$ or if $X_j'=e$}.
  \end{equation}
\item[Constraint Associativity] The following diagram commutes for each $i\in \{1,\ldots,n\}$ and objects $\ang{X} \in \txprod_j
    \C_j$, with $X_i', X_i'' \in \C_i$.
    \begin{equation}\label{eq:ml-f2-assoc}
    \begin{tikzpicture}[x=55mm,y=15mm,vcenter]
    \tikzset{0cell/.append style={nodes={scale=.85}}}
    \tikzset{1cell/.append style={nodes={scale=.85}}}
      \draw[0cell] 
      (0,0) node (a) {
        P\ang{X\compi X_i}
        \oplus P\ang{X\compi X_i'}
        \oplus P\ang{X\compi X_i''}
      }
      (35ex,0) node (b) {
        P\ang{X\compi X_i}
        \oplus P\ang{X\compi (X_i' \oplus X_i'')}
      }
      (0,-1) node (c) {
        P\ang{X\compi (X_i \oplus X_i')}
        \oplus P\ang{X\compi X_i''}
      }
      (b)+(0,-1) node (d) {
        P\ang{X\compi (X_i \oplus X_i' \oplus X_i'')}
      }
      ;
      \draw[1cell] 
      (a) edge node {1 \oplus P^2_i} (b)
      (a) edge['] node {P^2_i \oplus 1} (c)
      (b) edge node {P^2_i} (d)
      (c) edge node {P^2_i} (d)
      ;
    \end{tikzpicture}
    \end{equation}
\item[Constraint Symmetry] The following diagram commutes
    for each $i\in \{1,\ldots,n\}$ and objects $\ang{X} \in \txprod_j \C_j$, with
    $X_i' \in \C_i$.
    \begin{equation}\label{eq:ml-f2-symm}
    \begin{tikzpicture}[x=40mm,y=15mm,vcenter]
\tikzset{0cell/.append style={nodes={scale=.85}}}
\tikzset{1cell/.append style={nodes={scale=.9}}}
      \draw[0cell] 
      (0,0) node (a) {
        P\ang{X\compi X_i}
        \oplus P\ang{X\compi X_i'}
      }
      (1,0) node (b) {
        P\ang{X\compi (X_i \oplus X_i')}
      }
      (0,-1) node (c) {
        P\ang{X\compi X_i'}
        \oplus P\ang{X\compi X_i}
      }
      (1,-1) node (d) {
        P\ang{X\compi (X_i' \oplus X_i)}
      }
      ;
      \draw[1cell] 
      (a) edge node {P^2_i} (b)
      (a) edge['] node {\xi} (c)
      (b) edge node {P\ang{1 \compi \xi}} (d)
      (c) edge node {P^2_i} (d)
      ;
    \end{tikzpicture}
    \end{equation}
  \item[Constraint 2-By-2] The following diagram commutes
    for each
    \[
    i \not= k\in \{1,\ldots,n\}, 
    \quad
    \ang{X} \in \txprod_j \C_j,
    \quad
    X_i' \in \C_i,
    \andspace
    X_k' \in \C_k.
    \]
\begin{equation}\label{eq:f2-2by2}
\begin{tikzpicture}[x=25ex,y=11ex,vcenter]
\tikzset{0cell-nomm/.style={commutative diagrams/every diagram,every cell,nodes={scale=.8}}}
\tikzset{1cell/.append style={nodes={scale=.8}}}
      \draw[0cell-nomm] 
      (0,0) node[align=left] (a) {$
        \phantom{\oplus}P\ang{X\compi X_i \compk X_k}
        \oplus P\ang{X\compi X_i' \compk X_k}$\\       
        $\oplus P\ang{X\compi X_i \compk X_k'}
        \oplus P\ang{X\compi X_i' \compk X_k'}
      $}
      (0,-1) node[align=left] (a') {$
        \phantom{\oplus}P\ang{X\compi X_i \compk X_k}
        \oplus P\ang{X\compi X_i \compk X_k'}$\\
        $\oplus P\ang{X\compi X_i' \compk X_k}
        \oplus P\ang{X\compi X_i' \compk X_k'}
      $}
      (.7,.75) node (b) {$
        P\ang{X\compi (X_i \oplus X_i') \compk X_k}
        \oplus P\ang{X\compi (X_i \oplus X_i') \compk X_k'} 
      $}
      (.7,-1.75) node (b') {$
        P\ang{X\compi X_i \compk (X_k \oplus X_k')}
        \oplus P\ang{X\compi X_i' \compk (X_k \oplus X_k')}
      $}
      (1.1,-.5) node (c) {$
        P\ang{X\compi (X_i \oplus X_i') \compk (X_k \oplus X_k')}
      $}
      ;
      \draw[1cell] 
      (a) edge node[pos=.1] {P^2_i \oplus P^2_i} (b)
      (b) edge[transform canvas={shift={(.1,0)}}] node {P^2_k} (c)
      (a) edge[transform canvas={shift={(.05,0)}}] node[swap] {1 \oplus \xi \oplus 1} (a')
      (a') edge['] node[pos=.1] {P^2_k \oplus P^2_k} (b')
      (b') edge[',transform canvas={shift={(.1,0)}}] node {P^2_i} (c)
      ;
\end{tikzpicture}
\end{equation}
\end{description}
This finishes the definition of an $n$-linear functor.  

Moreover, we define the following.
\begin{itemize}
\item If $n=0$, then a 0-linear functor is a choice of an object in $\D$.
\item An $n$-linear functor $(P, \{P^2_j\})$ is \emph{strong} if each linearity constraint $P^2_j$ is a natural isomorphism.  It is \emph{strict} if each $P^2_j$ is an identity natural transformation.
\item A \emph{multilinear functor} is an $n$-linear functor for some $n \geq 0$.
\end{itemize}
Below, we express the domain of a multilinear functor $\P$ either as a product $\binprod_j \C_j$ or as a tuple $\ang{\C}$. So, for example, we write
\[
  \binprod_j \C_j \fto{\P} \D  \orspace \ang{\C} \fto{\P} \D
\]
to denote that $\P$ is a multilinear functor from $\ang{\C}$ to $\D$.
\end{definition}

\begin{example}
  For permutative categories $\C$ and $\D$, the definition of a 1-linear functor from $\C$ to $\D$ consists of the same data and axioms as the definition of a strictly unital symmetric monoidal functor from $\C$ to $\D$.
\end{example}

\begin{definition}\label{def:nlineartransformation}
Suppose $P,Q$ are $n$-linear functors as displayed below.
\begin{equation}\label{nlineartransformation}
\begin{tikzpicture}[xscale=3,yscale=2.5,baseline={(a.base)}]
\tikzset{0cell/.append style={nodes={scale=1}}}
\tikzset{1cell/.append style={nodes={scale=1}}}
\draw[0cell]
(0,0) node (a) {\prod_{j=1}^n \C_j}
(a)++(1,0) node (b) {\D}
;
\draw[1cell]  
(a) edge[bend left] node[pos=.45] {\big( P, \{P^2_j\}\big)} (b)
(a) edge[bend right] node[swap,pos=.45] {\big( Q, \{Q^2_j\}\big)} (b)
;
\draw[2cell] 
node[between=a and b at .47, shift={(0,0)}, rotate=-90, 2label={above,\theta}] {\Rightarrow}
;
\end{tikzpicture}
\end{equation}
An \emph{$n$-linear transformation} $\theta \cn P \to Q$ is a natural transformation of underlying functors that satisfies the following two \emph{multilinearity conditions}.
\begin{description}
\item[Unity] 
\begin{equation}\label{niitransformationunity}
\theta_{\ang{X}} = 1_e \qquad \text{if any $X_i = e \in \C_i$}.
\end{equation}
\item[Constraint Compatibility] The diagram
\begin{equation}\label{eq:monoidal-in-each-variable}
\begin{tikzpicture}[x=45mm,y=15mm,vcenter]
\tikzset{0cell/.append style={nodes={scale=.9}}}
\tikzset{1cell/.append style={nodes={scale=.9}}}
      \draw[0cell] 
      (0,0) node (a) {
        P\ang{X\compi X_i}
        \oplus P\ang{X\compi X_i'}
      }
      (1,0) node (b) {
        P\ang{X\compi (X_i \oplus X_i')}
      }
      (0,-1) node (c) {
        Q\ang{X\compi X_i}
        \oplus Q\ang{X\compi X_i'}
      }
      (1,-1) node (d) {
        Q\ang{X\compi (X_i \oplus X_i')}
      }
      ;
      \draw[1cell] 
      (a) edge node {P^2_i} (b)
      (a) edge['] node {\theta \oplus \theta} (c)
      (b) edge node {\theta} (d)
      (c) edge node {Q^2_i} (d)
      ;
  \end{tikzpicture}
\end{equation}
commutes for each $i \in \{1,\ldots,n\}$, $\ang{X} \in \txprod_j \C_j$, and $X_i' \in \C_i$.
\end{description}
A \emph{multilinear transformation} is an $n$-linear transformation for some $n \geq 0$.  Identities and composition of multilinear transformations are given componentwise.
\end{definition}

\begin{example}
  The definition of 1-linear transformation between 1-linear functors coincides with that of monoidal natural transformation between corresponding strictly unital symmetric monoidal functors.
\end{example}

\begin{definition}\label{definition:permcatsus-homcat}
  We define the following categories of multilinear functors and transformations.
  \begin{itemize}
  \item Let $\permcatsu\scmap{\ang{\C};\D}$ denote the category of multilinear functors from $\ang{\C}$ to $\D$ and multilinear transformations between them.
  \item Let $\permcatsus\scmap{\ang{\C};\D}$ denote the full subcategory of strong multilinear functors.
  \item Let $\permcatst\scmap{\ang{\C};\D}$ denote the full subcategory of strict multilinear functors.
  \end{itemize}
  For each of these, the 1-linear case coincides with the notation of \cref{definition:permcat}.
\end{definition}

Now we define the rest of the multicategory data for $\permcat$.
\begin{definition}[Symmetric Group Action]\label{definition:permcat-action}
  Suppose given multilinear functors $P$ and $Q$ in $\permcat\mmap{\D; \ang{\C}}$ together with a multinatural transformation $\theta$ as displayed below.
\begin{equation}\label{permiicatcddata}
\begin{tikzpicture}[xscale=3,yscale=2.5,vcenter]
\draw[0cell=.9]
(0,0) node (a) {\prod_{j=1}^n \C_j}
(a)++(1,0) node (b) {\D}
;
\draw[1cell=.9]  
(a) edge[bend left] node[pos=.45] {(P, \{P^2_j\})} (b)
(a) edge[bend right] node[swap,pos=.45] {(Q, \{Q^2_j\})} (b)
;
\draw[2cell] 
node[between=a and b at .5, shift={(0,0)}, rotate=-90, 2label={below,\theta}] {\Rightarrow}
;
\end{tikzpicture}
\end{equation}
For a permutation $\sigma \in \Sigma_n$, the symmetric group action 
\begin{equation}\label{permiicatsymgroupaction}
\begin{tikzcd}[column sep=large]
\permcat\mmap{\D; \ang{\C}} \ar{r}{\sigma} & \permcat\mmap{\D; \ang{\C}\sigma}
\end{tikzcd}
\end{equation}
sends the data \cref{permiicatcddata} to the following composites and whiskerings, where $\xi_\si$ denotes the isomorphism given by permutation of terms in the product.
\begin{equation}\label{permiicatsigmaaction}
\begin{tikzpicture}[xscale=3,yscale=2.5,vcenter]
\draw[0cell=.9]
(0,0) node (a) {\prod_{j=1}^n \C_j}
(a)++(1,0) node (b) {\D}
(a)++(-.6,0) node (c) {\prod_{j=1}^n \C_{\sigma(j)}}
;
\draw[1cell=.9]  
(c) edge node {\xi_\sigma} (a)
(a) edge[bend left] node[pos=.45] {(P, \{P^2_j\})} (b)
(a) edge[bend right] node[swap,pos=.45] {(Q, \{Q^2_j\})} (b)
;
\draw[2cell] 
node[between=a and b at .5, shift={(0,0)}, rotate=-90, 2label={below,\theta}] {\Rightarrow}
;
\end{tikzpicture}
\end{equation}
The $j$-th linearity constraint of $P^\sigma = P \circ \xi_\sigma$ is the composite in $\D$
\begin{equation}\label{fsigmatwoj}
\begin{tikzcd}[column sep=large, row sep=small]
P^{\sigma} \ang{A} \oplus P^{\sigma} \ang{A \compj A_j'} \ar{r}{(P^{\sigma})^2_j} \ar[equal]{d} & P^{\sigma} \ang{A \compj (A_j \oplus A_j')} \ar[equal]{d}\\
P\big( \sigma\ang{A} \big) \oplus P\big( \sigma\ang{A} \comp_{\sigma(j)} A_j' \big) \ar{r}{P^2_{\sigma(j)}} & P\big( \sigma\ang{A} \comp_{\sigma(j)} (A_j \oplus A_j') \big)
\end{tikzcd}
\end{equation}
for objects 
\[\ang{A} = (A_1,\ldots,A_n) \in \prod_{j=1}^n \C_{\sigma(j)} \andspace A_j' \in \C_{\sigma(j)}.\]  
Note that if $P$ is strong, respectively strict, with each $P^2_j$ a natural isomorphism, respectively identity, then $P^\sigma$ is also strong, respectively strict.
\end{definition}

\begin{definition}[Composition]\label{definition:permcat-comp}
  Suppose given, for each $j \in \{1,\ldots,n\}$,
  \begin{itemize}
  \item permutative categories $\ang{\B_j} = \big(\B_{j,1}, \ldots, \B_{j,k_j}\big)$,
  \item objects $P'_j$ and $Q'_j$ and a 1-cell $\theta_j$ in $\permcat\mmap{\C_j;\ang{\B_j}}$ 
  \end{itemize} 
  as follows.
  \begin{equation}\label{permiicatbcdata}
    \begin{tikzpicture}[xscale=3.5,yscale=3,vcenter]
      \draw[0cell=.9]
      (0,0) node (a) {\prod_{i=1}^{k_j} \B_{j,i}}
      (a)++(1,0) node (b) {\C_j}
      ;
      \draw[1cell=.9]  
      (a) edge[bend left] node[pos=.45] {P'_j} (b)
      (a) edge[bend right] node[swap,pos=.45] {Q'_j} (b)
      ;
      \draw[2cell] 
      node[between=a and b at .5, shift={(0,0)}, rotate=-90, 2label={below,\theta_j}] {\Rightarrow}
      ;
    \end{tikzpicture}
  \end{equation}
  With $\ang{\B} = \big(\ang{\B_1}, \ldots, \ang{\B_n}\big)$,
  the multicategorical composition functor
  \begin{equation}\label{permiicatgamma}
    \begin{tikzcd}[cells={nodes={scale=.9}}]
      \permcat\mmap{\D;\ang{\C}} \times \prod\limits_{j=1}^n \permcat\mmap{\C_j;\ang{\B_j}} \ar{r}{\gamma}
      & \permcat\mmap{\D;\ang{\B}}
    \end{tikzcd}
  \end{equation}
  sends the data \cref{permiicatcddata,permiicatbcdata} to the composites
  \begin{equation}\label{compositemodification}
    \begin{tikzpicture}[xscale=6.5,yscale=4,vcenter]
      \draw[0cell=.9]
      (0,0) node (a) {\prod_{j=1}^n \prod_{i=1}^{k_j} \B_{j,i}}
      (a)++(1,0) node (b) {\D}
      ;
      \draw[1cell=.9]  
      (a) edge[bend left] node[pos=.45] {P \circ \txprod_j P'_j} (b)
      (a) edge[bend right] node[swap,pos=.45] {Q \circ \txprod_j Q'_j} (b)
      ;
      \draw[2cell=.9] 
      node[between=a and b at .5, shift={(0,0)}, rotate=-90, 2label={below,\theta \otimes (\textstyle\prod_j \theta_j)}] {\Longrightarrow}
      ;
    \end{tikzpicture}
  \end{equation}
  defined as follows.

  \begin{description}
  \item[Composite Multilinear Functor]
  Suppose given tuples of objects 
  \begin{equation}\label{angwj}
    \begin{split}
      \ang{W_j} &= (W_{j,1},\ldots,W_{j,k_j}) \in \prod_{i=1}^{k_j} \B_{j,i} \qquad \text{for $j \in \{1,\ldots,n\}$ and}\\
      \ang{W} &= \big(\ang{W_1}, \ldots, \ang{W_n}\big) \in \prod_{j=1}^n \prod_{i=1}^{k_j} \B_{j,i}.
    \end{split}
  \end{equation} 
  Then we have the object
  \begin{equation}\label{compositefangw}
      \textstyle \big(P \comp \prod_j P_j'\big) \ang{W} = P\big(P_1'\ang{W_1}, \ldots, P_n'\ang{W_n}\big) \inspace \D.
  \end{equation}

  To describe the linearity constraints of the composite $P \circ \prod_j P'_j$ in \cref{compositemodification}, in addition to the objects in \cref{angwj}, consider
  \begin{itemize}
  \item an object $W_{j,i}' \in \B_{j,i}$ for some choice of $(j,i)$ with $\ell = k_1 + \cdots + k_{j-1} + i$ and
  \item $\ang{P'W} = \big(P'_1\ang{W_1}, \ldots, P'_n\ang{W_n}\big) \in \prod_{j=1}^n \C_j$.
  \end{itemize} 
  Note that 
  \[\begin{split}
  \ang{W_j \compi W_{j,i}'} &= \big(\overbracket[.5pt]{W_{j,1},\ldots, W_{j,i-1}}^{\text{empty if $i=1$}},\, W_{j,i}', \overbracket[.5pt]{W_{j,i+1},\ldots, W_{j,k_j}}^{\text{empty if $i=k_j$}}\big)\\
  \ang{W_j \compi (W_{j,i} \oplus W_{j,i}')} &= \big(\underbracket[.5pt]{W_{j,1},\ldots, W_{j,i-1}}_{\text{empty if $i=1$}},\, W_{j,i} \oplus W_{j,i}', \underbracket[.5pt]{W_{j,i+1},\ldots, W_{j,k_j}}_{\text{empty if $i=k_j$}}\big).
  \end{split}\]
  The $\ell$-th linearity constraint $\big(P \circ \txprod_j P'_j\big)^2_\ell$ is defined as the following composite in $\D$.
  \begin{equation}\label{ffjlinearity}
    \begin{tikzpicture}[xscale=3,yscale=1,vcenter]
      \def\h{1} \def\v{-1}
      \draw[0cell=.75] 
      (0,0) node (x0) {P\ang{P'W} \oplus P\ang{P'W \compj P'_j\ang{W_j \compi W_{j,i}'}}}
      (x0)++(-\h,\v) node (x1) {(P \circ \txprod_j P'_j)\ang{W} \oplus (P \circ \txprod_j P'_j)\ang{W \comp_\ell W_{j,i}'}}
      (x1)++(1.3*\h,\v) node (x2) {P\ang{P'W \compj (P'_j\ang{W_j} \oplus P'_j\ang{W_j \compi W_{j,i}'})}} 
      (x1)++(0,2*\v) node (x3) {(P \circ \txprod_j P'_j)\ang{W \comp_\ell (W_{j,i} \oplus W_{j,i}')}}
      (x3)++(\h,\v) node (x4) {P\ang{P'W \compj P'_j\ang{W_j \compi (W_{j,i} \oplus W_{j,i}')}}}
      ;
      \draw[1cell=.75] 
      (x1) edge[-,double equal sign distance] (x0)
      (x0) edge node[pos=.7] {P^2_j} (x2)
      (x2) edge node[pos=.3] {P\ang{1 \compj (P'_j)^2_i}} (x4)
      (x1) edge node[swap] {\big(P \circ \txprod_j P'_j\big)^2_\ell} (x3)
      (x3) edge[-,double equal sign distance] (x4)
      ; 
    \end{tikzpicture}
  \end{equation}
  Note that if $P$ and each $P'_j$ are strong, respectively, strict, then each linearity constraint $\big(P \circ \txprod_j P'_j\big)^2_\ell$ is componentwise invertible, respectively, an identity, and the composite $P \circ \txprod_j P'_j$ is also strong, respectively, strict.

  \item[Composite Multinatural Transformation]
  The multinatural transformation $\theta \otimes \big(\prod_j \theta_j\big)$ in \cref{compositemodification} is the horizontal composite of the natural transformations $\prod_j \theta_j$ and $\theta$.
  The component morphism $\big(\theta \otimes (\prod_j \theta_j)\big)_{\angW}$ is the composite
  \begin{equation}\label{thetaprodthetaw}
    \begin{tikzpicture}[xscale=3.5,yscale=3.5,vcenter]
      \draw[0cell=.9]
      (0,0) node (a) {P \ang{P_j'\angWj}_j}
      (a)++(1.1,0) node (b) {P \ang{Q_j'\angWj}_j}
      (b)++(1,0) node (c) {Q \ang{Q_j'\angWj}_j}
      ;
      \draw[1cell=.9]  
      (a) edge node {P\ang{(\theta_j)_{\angWj}}_j} (b)
      (b) edge node {\theta_{\ang{Q_j'\angWj}_j}} (c)
      ;
    \end{tikzpicture}
  \end{equation}
  in $\D$.\dqed
  \end{description}
\end{definition}

The following construction of pointed multicategories from permutative categories leads to an alternative description of multilinearity via the smash product of pointed multicategories.
\begin{definition}\label{definition:EndC-basepoint}
  Suppose $\C$ is a small permutative category.
  The \emph{endomorphism multicategory} $\End(\C)$ is the small multicategory with object set $\Ob\C$ and with
  \[\End(\C)\mmap{Y;\ang{X}} = \C(X_1 \oplus \cdots \oplus X_n , Y)\]
  for $Y \in \Ob\C$ and $\ang{X}=(X_1,\cdots,X_n) \in (\Ob\C)^{\times n}$.  An empty $\oplus$ means the unit object $e$.

  The \emph{canonical basepoint} of $\End(\C)$ is determined by the multifunctor
  \[
  \Mterm \to \End(\C)
  \]
  sending the single object of $\Mterm$ to the monoidal unit $e$ in $\C$ and the $n$-ary operation $\iota_n$ to the identity morphism of
  \[
  \bigoplus_{i=1}^n e = e.\dqed
  \]

  Strictly unital monoidal functors, and monoidal natural transformations between them, induce pointed multifunctors and pointed multinatural transformations, respectively, on endomorphism multicategories.  This defines a 2-functor
  \[
    \End\cn \permcatsu \to \Multicat
  \]
  by \cite[5.3.6]{cerberusIII}.  Equipped with canonical basepoints, $\End$ takes values in $\Multicat_*$.
\end{definition}

The following result identifies multilinear functors and transformations between permutative categories with multifunctors and multinatural transformations of endomorphism categories.
\begin{proposition}[{\cite[6.5.10~and~6.5.13]{cerberusIII}}]\label{proposition:n-lin-equiv}
  For permutative categories $\C_1,\ldots,\C_n$, and $\D$, the 2-functor $\End$ induces an isomorphism of categories
  \begin{align*}\label{eq:multilin-via-sma}
    \End \cn \permcatsu\scmap{\ang{\C};\D} \fto{\iso}
    & \Multicat_*\scmap{\ang{\End(\C)};\End(\D)}\\
    & = \Multicat_*\big(\sma_{i} \End(\C_i),\End(\D)\big),
  \end{align*}
  between the category of $n$-linear functors and transformations $\ang{\C} \to \D$ and the category of pointed multifunctors and multinatural transformations $\sma_i \End(\C_i) \to \End(\D)$.
\end{proposition}

\begin{definition}\label{definition:permcat-multicat}
  Let $\permcatsu$ denote the $\Cat$-enriched multicategory whose category of $n$-ary operations
  \[
  \permcatsu\mmap{\D;\ang{\C}}
  \]
  is the category of $n$-linear functors and $n$-linear transformations
  \[
  \ang{\C} \to \D.
  \]
  The symmetric group action and composition in $\permcatsu$ are those of \cref{definition:permcat-action,definition:permcat-comp}, respectively.
  The multicategory axioms for $\permcatsu$ follow from \cref{proposition:n-lin-equiv} and the symmetric monoidal axioms of the smash product.
  Independently, a direct verification is given in \cite[Section~6.6]{cerberusIII}.

  We let $\permcatsus$ denote the $\Cat$-enriched sub-multicategory
  whose $n$-ary operations consist of strong $n$-linear functors and
  $n$-linear transformations.
\end{definition}

\begin{remark}
  The multicategory structure on $\permcatsu$ is not induced by a symmetric monoidal structure.
  For example, the unit for the smash product of multicategories is not a permutative category.
  See \cite[Propositions~5.7.23 and~10.2.17]{cerberusIII} for further discussion of this point.
\end{remark}

\begin{remark}[Strict Unity]\label{remark:permcatsu}
In the rest of this paper, we mainly work with $\permcatsu$ instead of other variants in \cref{definition:permcat,definition:permcatsus-homcat}.  We now briefly discuss the technical advantages of $\permcatsu$ over other variants.
\begin{enumerate}
\item\label{rk:permcatsu-i} For a symmetric monoidal functor $P \cn \C \to \D$ between small permutative categories, the endomorphism multifunctor
\[\End(P) \cn \End(\C) \to \End(\D)\]
is \emph{not} a pointed multifunctor in general, where $\End(\C)$ and $\End(\D)$ are equipped with the canonical basepoints given by their respective monoidal units.  The multifunctor $\End(P)$ is pointed precisely when $P$ is strictly unital.  The $\Cat$-multicategory structure on $\permcatsu$ is canonically induced by the one on $\Multicat_*$ via $\End$, as in \cref{proposition:n-lin-equiv}.  In this sense, $\permcatsu$ is more convenient than $\permcat$.
\item\label{rk:permcatsu-ii} As we recall in \cref{proposition:free-perm-functor} below, the free permutative category 2-functor $\bigf$ has codomain the 2-category $\permcatst$, with strict symmetric monoidal functors as 1-cells.  However, in order to extend $\bigf$ to a non-symmetric $\Cat$-multifunctor (\cref{theorem:F-multi}), we will need to precompose the assignment of $\bigf$ on 1-cells and 2-cells with a \emph{strong} multilinear functor $S$; see \cref{eq:FbarS}.  While $S$ (\cref{definition:S-multi}) is a strong multilinear functor, it is not strict because its linearity constraints, defined in \cref{eq:S2b}, are generally not identity morphisms.  Thus, the non-symmetric $\Cat$-multicategorical extension of $\bigf$ should not use the codomain $\permcatst$, since we must allow non-identity linearity constraints.
\item\label{rk:permcatsu-iii} Segal's $K$-theory functor \cite{segal}
\[\Kse \cn \permcatsu \to \SymSp\]
is most naturally defined on $\permcatsu$; see \cite[Chapter 8]{cerberusIII} for a detailed discussion of $\Kse$.  In \cref{sec:hty-thy-equiv} we will use $\Kse$ to define stable equivalences in $\permcatsu$ and $\Multicat$.   The functor $\Kse$ is a composite of three functors, the first of which, called Segal $J$-theory and denoted $\Jse$, lands in the category $\Gacat$ of $\Ga$-categories.  For a small permutative category $\C$, each level of the $\Ga$-category $\Jse\C$ is a category in which an object is a system of objects $\{C_s\}$ in $\C$ along with some gluing morphisms, satisfying several axioms.  Among these axioms is an object unity axiom---see \cite[(8.3.2)]{cerberusIII}---that says 
\[C_\emptyset = e,\]
the monoidal unit in $\C$.  For $P$ as in \cref{rk:permcatsu-i} above, in order for $P$ to induce a morphism of $\Ga$-categories
\[\Jse P \cn \Jse\C \to \Jse\D,\]
we need the condition
\[P(C_\emptyset) = Pe = e,\]
the monoidal unit in $\D$.  In other words, we need $P$ to be strictly unital.\defmark
\end{enumerate}
\end{remark}

\section{Free Permutative Category on a Multicategory}\label{sec:free-perm}

In this section we recall from \cite[Section~5]{johnson-yau-permmult} the free construction
\[
\bigf \cn \Multicat \to \permcatsu.
\]
In \cref{sec:cat-multi-f} we show that $\bigf$ is a non-symmetric $\Cat$-enriched multifunctor.
The definition of $\bigf$ makes use of sequences $\ang{x}$, indexing functions $f$, and permutations $\si^k_{g,f}$.
We make the following preliminary definitions.
\begin{definition}\label{definition:ufs}
  For a natural number $r \ge 0$ we let
  \[
    \ufs{r} = \{1, \ldots, r\}
  \]
  denote the finite set with $r$ elements.
\end{definition}
\begin{definition}\label{definition:free-perm-helper}
  Suppose $\M$ is a multicategory, and suppose $\ang{x}$ is a sequence of length $r$, with each $x_i \in \M$.
  Suppose
  \[
  f\cn \ufs{r} \to \ufs{s} \andspace g\cn \ufs{s} \to \ufs{t}
  \]
  are functions of finite sets, for $r,s,t \ge 0$.
  Then we define the following.
  \begin{itemize}
  \item For $j \in \ufs{s}$, let
    \begin{equation}\label{eq:x-finv}
      \ang{x}_{f^\inv(j)} = \ang{x_i}_{i \in f^\inv(j)}
    \end{equation}
    denote the sequence formed by those $x_i$ with $i \in f^\inv(j)$, ordered as in $\ang{x}$.
    Similarly, for a length-$s$ sequence of operations $\ang{\phi}$ in $\M$ and $k \in \ufs{t}$, let
    \[
    \ang{\phi}_{g^\inv(k)} = \ang{\phi_j}_{j \in g^\inv(k)}.
    \]
  \item For $k \in \ufs{t}$, let $\si^k_{g,f} \in \Si_t$ be the unique permutation such that
    \begin{equation}\label{eq:sigma-kgf}
      \bigg[
      \bigoplus_{j \in g^\inv(k)} \ang{x}_{f^\inv(j)}
      \bigg]
      \cdot \sigma^k_{g,f} 
      =
      \ang{x}_{(gf)^\inv(k)},
    \end{equation}
    where the sequence on the left hand side is the concatenation of sequences in the order specified by $g^\inv(k)$.
    We will use the action of these permutations on both objects and operations.
    \dqed
  \end{itemize}
\end{definition}

\begin{definition}\label{definition:free-perm}
  Suppose $\M$ is a multicategory.  Define a permutative category $\bigf\M$, called the \emph{free permutative category on $\M$}, as follows.
  \begin{description}
  \item[Objects] The objects of $\bigf\M$ are given by the $(\Ob\M)$-profiles: finite ordered sequences $\ang{x} = (x_1,\ldots,x_r)$ of objects of $\M$, with $r \ge 0$.
  \item[Morphisms] Given sequences $\ang{x}$ and $\ang{y}$ with lengths $r$ and $s$, respectively, the morphisms from $\ang{x}$ to $\ang{y}$ in $\bigf\M$ are given by pairs $(f,\ang{\phi})$ consisting of
    \begin{itemize}
    \item a function
      \[
      f \cn \ufs{r} \to \ufs{s}
      \]
      called the \emph{index map} and
    \item an ordered sequence of operations
      \[
      \ang{\phi} \withspace \phi_j \in \M\mmap{y_j;\ang{x}_{f^\inv(j)}}
      \]
      for $j \in \ufs{s}$, with $\ang{x}_{f^\inv(j)}$ defined in \cref{eq:x-finv}.
    \end{itemize}
    The identity morphism on $\ang{x}$ is given by $1_{\ufs{r}}$ and the tuple of unit operations $1_{x_i}$.

  \item[Composition] The composition of a pair of morphisms
    \[
    \ang{x} \fto{(f,\ang{\phi})} \ang{y} \fto{(g,\ang{\psi})} \ang{z}
    \]
    is the pair
    \begin{equation}\label{eq:FM-comp}
    \big(gf , \ang{\theta_k \cdot \si^k_{g,f}}_{k \in \ufs{t}} \big),
    \end{equation}
    where, for each $k \in \ufs{t}$,
    \begin{equation}\label{eq:thetak}
      \theta_k = \ga\lrscmap{\psi_k;\ang{\phi}_{g^\inv(k)}} \in \M\lrscmap{\;\bigoplus_{j \in g^\inv(k)} \ang{x}_{f^\inv(j)} ; z_k}.
    \end{equation}
    Note that the input profile for $\theta_k$ is the concatenation of $\ang{x}_{f^\inv(j)}$ for $j \in g^\inv(k)$.
    By definition \cref{eq:sigma-kgf}, the right action of $\si^k_{g,f}$ permutes this input profile to $\ang{x}_{(gf)^\inv(k)}$.
    Composition of morphisms is verified to be unital and associative in \cite[Proposition~5.7]{johnson-yau-permmult}.

  \item[Monoidal Sum] The monoidal sum 
    \[
      \oplus \cn \bigf\M \times \bigf\M \to \bigf\M
    \]
    is given on objects by concatenation of sequences.
    The monoidal sum of morphisms
    \[
      (f,\ang{\phi}) \cn \ang{x} \to \ang{y}
      \andspace
      (f',\ang{\phi'}) \cn \ang{x'} \to \ang{y'},
    \]
    is the pair
    \[
      \big(f \oplus f', \ang{\phi} \oplus \ang{\phi'}\big)
    \]
    where $f \oplus f'$ denotes the composite
    \[
      \ufs{r + r'} \iso \ufs{r} \bincoprod \ufs{r'} \fto{f \bincoprod f'} \ufs{s} \bincoprod \ufs{s'} \iso \ufs{s + s'}
    \]
    given by the disjoint union of $f$ with $f'$ and the canonical order-preserving isomorphisms.
    Functoriality of the monoidal sum follows because disjoint union of indexing functions preserves preimages and the operations in a composite \cref{eq:FM-comp} are determined elementwise for the indexing set of the codomain.

  \item[Monoidal Unit] The monoidal unit is the empty sequence $\ang{}$. 
    The unit and associativity isomorphisms for $\oplus$ are identities.  
    
  \item[Symmetry] The symmetry isomorphism for sequences $\ang{x}$ of length $r$ and $\ang{x'}$ of length $r'$ is 
    \begin{equation}\label{eq:FM-xi}
      \xi_{\ang{x},\ang{x'}} = \big(\tau_{r,r'} , \ang{1}\big)
    \end{equation} 
    where
    \[
      \tau_{r,r'} \cn
      \ufs{r + r'} \iso \ufs{r} \bincoprod \ufs{r'} \to \ufs{r'} \bincoprod \ufs{r}
      \iso \ufs{r' + r}
    \]
    is induced by the block-transposition of $\ufs{r}$ with $\ufs{r'}$, keeping the relative order within each block fixed.

    Concatenation of sequences is strictly associative and unital.
    The symmetry and hexagon axioms \cref{symmoncatsymhexagon} follow from the corresponding equalities of block permutations.
\end{description}
This completes the definition of $\bigf\M$.
\end{definition}

\begin{definition}\label{definition:free-smfun}
  Suppose $H\cn \M \to \N$ is a multifunctor.
  Define a strict symmetric monoidal functor
  \[
  \bigf H \cn \bigf\M \to \bigf\N
  \]
  via the following assignment on objects and morphisms.  For a sequence $\ang{x}$ of length $r$, define
  \[
  (\bigf H)\ang{x} = \ang{Hx_i}_{i \in \ufs{r}}.
  \]
  For a morphism $(f,\ang{\phi})$, define
  \begin{equation}\label{eq:FH-fphi}
  (\bigf H)(f,\ang{\phi}) = (f,\ang{H\phi_j}_{j}).
  \end{equation}

  The multifunctoriality of $H$ shows that this assignment is functorial on morphisms.
  Since the monoidal sum is defined by concatenation in $\bigf\M$ and $\bigf\N$, the functor $\bigf H$ is strict monoidal.
  Compatibility with the symmetry of $\bigf\M$ and $\bigf\N$ follows because $\bigf H$ preserves the index map of each morphism and $H$ preserves unit operations.
\end{definition}

\begin{definition}\label{definition:free-perm-multinat}
  Suppose $\kappa\cn H \to K \cn \M \to \N$ is a multinatural transformation.
  Define a monoidal natural transformation
  \[
  \bigf\kappa \cn \bigf H \to \bigf K
  \]
  via components
  \begin{equation}\label{eq:Fka-x}
  (\bigf\kappa)_{\ang{x}} = (1, \ang{\ka_{x_i}}_i) \cn \ang{Hx} \to \ang{Kx} 
  \end{equation}
  for each sequence $\ang{x}$ in $\bigf\M$.
  Naturality of $\bigf\kappa$ follows from multinaturality of $\kappa$ \cref{enr-multinat} because each $\si^j_{f,1}$ and $\si^j_{1,f}$ is an identity permutation and we have
  \begin{align*}
  (1,\ang{\ka_{y_j}}_j)(f,\ang{H\phi_j}_j)
    & = \bigg(f, \big\langle
      \ga\scmap{\ka_{y_j} ; H\phi_j}
    \big\rangle_{j} \bigg)\\
    & = \bigg(f, \big\langle
      \ga\scmap{K\phi_j ; \ang{\ka_{x_i}}_{i \in f^\inv(j)}}
    \big\rangle_j\bigg)\\
    & = (f,\ang{K\phi_j}_j)(1,\ang{\ka_{x_i}}_i)
  \end{align*}
  for each morphism $(f,\ang{\phi})\cn \ang{x} \to \ang{y}$ in $\bigf\M$.

  The monoidal naturality axioms \cref{eq:monoidal-nt-constr} for $\bigf\ka$ follow because the monoidal sum in $\bigf\N$ is given by concatenation of object and operation sequences.  The component $(\bigf\ka)_{\ang{}}$ is the identity morphism $(1_{\varnothing},\ang{})\cn \ang{} \to \ang{}$.
\end{definition}

\begin{proposition}[{\cite[Proposition~5.13]{johnson-yau-permmult}}]\label{proposition:free-perm-functor}
  The free permutative category construction given in \cref{definition:free-perm,definition:free-smfun} provides a 2-functor
  \[
  \bigf \cn \Multicat \to \permcatst.
  \]
\end{proposition}

\begin{example}\label{example:free-terminal}
  For the terminal multicategory $\Mterm$, the free permutative category $\bigf\Mterm$ is isomorphic to the natural number category $\N$ whose objects are given by natural numbers and morphisms are given by morphisms of finite sets
  \[
  \N(r,s) = \Set(\ufs{r},\ufs{s}).
  \]
  The natural number $r \in \N$ corresponds to the length-$r$ sequence whose terms are the unique object of $\Mterm$.  Each morphism $f\cn \ufs{r} \to \ufs{s}$ corresponds to the morphism
  \[
  (f,\ang{\phi}) \in \bigf\Mterm
  \]
  where $\phi_j$ is the unique operation in $\Mterm$ of arity $|f^\inv(j)|$.
\end{example}

\begin{example}\label{example:free-Mtu}
  For the initial operad $\Mtu$, the free permutative category $\bigf\Mtu$ is isomorphic to the permutation category $\Si$ with objects given by natural numbers and morphisms given by permutations
  \[
  \Si(r,s) = \begin{cases}
    \Si_r, & \ifspace r = s,\\
    \varnothing, & \ifspace r \not= s
  \end{cases}
  \]
  for each pair of natural numbers $r$ and $s$.
\end{example}

\section{Assignment on Multimorphism Categories}\label{sec:multimorphism-assignment}

Throughout this section we suppose $n \ge 0$ and consider small multicategories $\M_1$, \ldots, $\M_n$, and $\N$.  The purpose of this section is to define the assignment on multimorphism categories
\[
\bigf\cn \Multicat\scmap{\ang{\M};\N} \to \permcatsu\scmap{\ang{\bigf\M};\bigf\N}.
\]
This assignment is provided by a strong $n$-linear functor of permutative categories
\[
S\cn \ang{\bigf\M} \to \bigf\big(\bigotimes_{i=1}^n \M_i\big)
\]
and the 2-functoriality of $\bigf$.  We will use the following notation in the definition of $S$ below, in the case $n > 0$.
\begin{definition}\label{definition:xonen}
  Suppose given objects
  \[
  \ang{x^i} \in \bigf\M_i \forspace 1 \leq i \leq n,
  \]
  where each $\ang{x^i} = (x^i_1,\ldots,x^i_{r_i})$.  So $\ang{x^i}$ has length $r_i$ and each $x^i_j$ is an object of $\M_i$.  For each $n$-tuple of indices $(\jonejn)$ with $1 \leq j_i \leq r_i$, define
  \[
  x^{\onen}_{\jonejn} = \big(x^1_{j_1}, x^2_{j_2}, \ldots, x^n_{j_n}\big) \in \bigotimes_{i=1}^n \M_i.
  \]
  Then define an object
  \[
  \ang{x^{\onen}} \in \bigf\big(\bigotimes_{i=1}^n \M_i \big)
  \]
  of length $r_{\onen}$, where
  \begin{align}
    r_{\onen} & = \prod_{i=1}^n r_i \andspace\\
    \ang{x^{\onen}} & = \Big\langle\cdots\Big\langle x^{\onen}_{\jonejn} \Big\rangle_{j_1=1}^{r_1} \cdots \Big\rangle_{j_n = 1}^{r_n}.\label{eq:xonen}
  \end{align}
  Using the tensor product of profiles from \cref{definition:xiotimes}, the tuple $\ang{x^\onen}$ from \cref{eq:xonen} is the iterated tensor product of the tuples $\ang{x^i}$, with the reverse lexicographic ordering in the subscripts.
\end{definition}

\begin{example}\label{ex:angxonen}
To illustrate \cref{definition:xonen}, consider the case with $n=3$ and the following objects.
\begin{equation}\label{xonextwoxthree}
\begin{split}
\ang{x^1} &= (x^1_1, x^1_2) \in \bigf \M_1\\
\ang{x^2} &= (x^2_1, x^2_2, x^2_3) \in \bigf \M_2\\
\ang{x^3} &= (x^3_1, x^3_2) \in \bigf \M_3
\end{split}
\end{equation}
For each triple of indices $(j_1,j_2,j_3)$ with $j_1 \in \{1,2\}$, $j_2 \in \{1,2,3\}$, and $j_3 \in \{1,2\}$, we have the object
\[x^{123}_{j_1,j_2,j_3} = \big(x^1_{j_1}, x^2_{j_2}, x^3_{j_3}\big) \in \M_1 \otimes \M_2 \otimes \M_3.\]
This object is obtained from the array \cref{xonextwoxthree} by forming a vertical product, using the $j_1$-th object in the first row, the $j_2$-th object in the second row, and the $j_3$-th object in the third row.  For example, for the indices $(j_1,j_2,j_3) = (2,3,1)$, we have the object
\[x^{123}_{231} = (x^1_2, x^2_3, x^3_1) \in \M_1 \otimes \M_2 \otimes \M_3.\]
The object in \cref{eq:xonen} 
\[\ang{x^{123}} \in \bigf(\M_1 \otimes \M_2 \otimes \M_3)\]
is the following sequence of $2 \cdot 3 \cdot 2 = 12$ objects, with each object in $\M_1 \otimes \M_2 \otimes \M_3$, read left to right in the first row and then the second row.
\[\begin{split}
&\overbracket[.5pt]{x^{123}_{111} \quad x^{123}_{211}}^{j_2 \,=\, 1} \quad 
\overbracket[.5pt]{x^{123}_{121} \quad x^{123}_{221}}^{j_2 \,=\, 2} \quad 
\overbracket[.5pt]{x^{123}_{131} \quad x^{123}_{231}}^{j_2 \,=\, 3} \qquad j_3=1\\
&\underbracket[.5pt]{x^{123}_{112} \quad x^{123}_{212}}_{j_2 \,=\, 1} \quad 
\underbracket[.5pt]{x^{123}_{122} \quad x^{123}_{222}}_{j_2 \,=\, 2} \quad 
\underbracket[.5pt]{x^{123}_{132} \quad x^{123}_{232}}_{j_2 \,=\,3} \qquad j_3=2
\end{split}\]
In the previous display, the two rows correspond to $j_3 = 1$ and 2, as indicated by the last subscripts.  In each row, the three pairs from left to right correspond to $j_2 = 1,2$, and 3, as indicated by the middle subscripts.  Within each pair, the two objects correspond to $j_1 =1$ and 2, as indicated by the first subscript. 
\end{example}

\begin{definition}\label{definition:phionen}
  Suppose given objects and morphisms
  \[
  (f^i,\ang{\phi^i}) \cn \ang{x^i} \to \ang{y^i} \inspace \bigf\M_i \forspace 1 \leq i \leq n,
  \]
  where
  \begin{itemize}
  \item each $\ang{x^i}$ has length $r_i$,
  \item each $\ang{y^i}$ has length $s_i$,
  \item each $f^i \cn \ufs{r_i} \to \ufs{s_i}$, and
  \item each $\phi^i_j \in \M_i\scmap{\ang{x^i}_{(f^i)^\inv(j)}; y^i_j}$.
  \end{itemize}
  Define $f^{\onen}$ as the composite below, where the unlabeled isomorphisms are given by the reverse lexicographic ordering of the products:
  \begin{equation}\label{eq:fonen}
  \begin{tikzpicture}[x=20mm,y=20mm,vcenter]
    \draw[0cell] 
    (0,0) node (a) {\ufs{r_\onen}}
    (1,0) node (b) {\binprod_i \ufs{r_i}}
    (2,0) node (c) {\binprod_i \ufs{s_i}}
    (3,0) node (d) {\ufs{s_\onen}.}
    ;
    \draw[1cell] 
    (a) edge node {\iso} (b)
    (c) edge node {\iso} (d)
    (b) edge node {\binprod_i f^i} (c)
    ;
    \draw[1cell]
    (a.-90) -- +(0,-4ex) -| node[xarrow,pos=.25] {f^\onen} (d.-90)
    ;
  \end{tikzpicture}
  \end{equation} 
  For each $n$-tuple of indices $(\konekn)$ with $1 \leq k_i \leq s_i$, let
  \[
  \ang{x^{\onen}}_{f;\,\konekn} = \Big\langle\cdots\Big\langle x^\onen_{\jonejn} \Big\rangle_{j_1 \in (f^1)^\inv(k_1)} \cdots \Big\rangle_{j_n \in (f^n)^\inv(k_n)}
  \]
  and define
  \begin{equation}\label{eq:phionenkonekn}
    \phi^{\onen}_{\konekn} \cn \ang{x^{\onen}}_{f;\,\konekn} \to y^{\onen}_{\konekn}
  \end{equation} 
  as the tensor product $\otimes_{i=1}^n \phi^i_{k_i}$.  Then define
  \begin{equation}\label{eq:phionen}
    \ang{\phi^{\onen}} = \Big\langle\cdots\Big\langle \phi^\onen_{\konekn} \Big\rangle_{k_1 = 1}^{s_1} \cdots \Big\rangle_{k_n = 1}^{s_n}.
  \end{equation}
  This defines a morphism
  \[
    (f^\onen,\ang{\phi^\onen}) \cn \ang{x^\onen} \to \ang{y^\onen}
  \]
  in $\bigf\big(\bigotimes_{i=1}^n\M_i\big)$.
\end{definition}

Recall the initial operad $\Mtu$ (\cref{definition:IT}), which is the monoidal unit for the tensor product of small multicategories (\cref{theorem:Multicat-sm}).  Also recall from \cref{example:free-Mtu} that $\bigf(\Mtu)$ is the permutation category $\Si$.
\begin{definition}[Multilinear $S$]\label{definition:S-multi}
  Suppose $n \ge 0$ and suppose given small multicategories $\M_1$, \ldots, $\M_n$.  Define an $n$-linear functor
  \[
    (S,S^2_b)\cn \prod_{i=1}^n \bigf\M_i \to \bigf\big( \bigotimes_{i=1}^n \M_i \big)
  \]
  as follows.

  For $n = 0$, we define $S$ by choice of object $1 \in \Sigma$.  For $n > 0$, we make the following definitions.  Suppose given objects and morphisms
  \[
    (f^i,\ang{\phi^i}) \cn \ang{x^i} \to \ang{y^i} \inspace \bigf\M_i \forspace 1 \leq i \leq n,
  \]
  as in \cref{definition:phionen}.
  \begin{description}
  \item[Underlying Functor]
    The underlying functor $S$ is given by the following assignments, using \cref{eq:xonen,eq:fonen,eq:phionen}:
    \begin{align}
      S\big(\ang{x^1},\ldots,\ang{x^n}\big)
      & = \ang{x^\onen} \andspace\label{eq:Sxi}\\
      S\big((f^1,\ang{\phi^1}),\ldots,(f^n,\ang{\phi^n})\big)
      & = \big(f^\onen,\ang{\phi^\onen}\big).\label{eq:Sfi}
    \end{align}

  \item[Linearity Constraints]
    Suppose $1 \leq b \leq n$ and suppose $\ang{\hat{x}^b}$ is an object in $\bigf\M_b$ with length $\hat{r}_b$.  Let
    \[
      \tilde{r}_b = r_b + \hat{r}_b \andspace \ang{\tilde{x}^b} = \ang{x^b} \oplus \ang{\hat{x}^b}.
    \]
    Then define
    \[
      \ang{\hat{x}^\onen} \andspace \ang{\tilde{x}^\onen}
    \]
    as in \cref{eq:xonen}, using $\ang{\hat{x}^b}$ and $\ang{\tilde{x}^b}$, respectively, in place of $\ang{x^b}$.

    The $b$th linearity constraint, $S^2_b$, is defined by components
    \begin{equation}\label{eq:S2b}
      S^2_b = (\rho_{r_b,\hat{r}_b}, \ang{1}) \cn \ang{x^\onen} \oplus \ang{\hat{x}^\onen} \to \ang{\tilde{x}^\onen}
    \end{equation} 
    where $\ang{1}$ is the tuple of identity operations and $\rho_{r_b,\hat{r}_b}$ is the unique permutation of entries determined by the source and target of $S^2_b$.  Naturality of $S^2_b$ with respect to morphisms in $\prod_i \bigf\M_i$ follows from the uniqueness of $\rho_{r_b,\hat{r}_b}$.
  \end{description}
  This finishes the definition of $S$ as an assignment on objects and morphisms, and the definition of natural transformations $S^2_b$.
  We verify that $S$ is functorial in \cref{proposition:functoriality-S}.  We verify the multilinearity axioms of \cref{def:nlinearfunctor} in \cref{proposition:multilinearity-S}.
\end{definition}

\begin{proposition}\label{proposition:functoriality-S}
  In the context of \cref{definition:S-multi}, $S$ is a functor.
\end{proposition}
\begin{proof}
  If each $(f^i,\ang{\phi^i})$ is an identity, then so are $f^\onen$ and each $\ang{\phi^\onen}$.  Therefore $S$ preserves identities.
  
  Suppose given composable morphisms
  \[
    \ang{x^i} \fto{(f^i,\ang{\phi^i})} \ang{y^i}
    \fto{(g^i, \ang{\psi^i})} \ang{z^i}
    \inspace \bigf\M_i
  \]
  for $1 \leq i \leq n$.  Then by \cref{eq:FM-comp} and \cref{eq:thetak} we have
  \[
    (g^i, \ang{\psi^i}) \circ (f^i, \ang{\phi^i}) =
    (g^i f^i, \ang{\theta^i_{\ell_i} \cdot \sigma^{\ell_i}_{g^i,f^i}}_{\ell_i}),
    \withspace
    \theta^i_{\ell_i} = \ga\scmap{\psi^i_{\ell_i}; \ang{\phi^i}_{(g^i)^\inv(\ell_i)}}.
  \]
  Applying $S$ to the tuple of these composites, we have
  \[
    S\big( \binprod_i (g^if^i, \ang{\theta^i_{\ell_i} \cdot \sigma^{\ell_i}_{g^i,f^i}}_{\ell_i} ) \big) =
    \big( h^{\onen}, \omega^{\onen} \big),
  \]
  where $h^i = g^i f^i$ and
  \[
    \om^{\onen}_{\elloneelln} = \otimes_i \big(\theta^i_{\ell_i} \cdot \si^{\ell_i}_{g^i,f^i}\big).
  \]

  Alternatively, applying $S$ and then composing results in the following:
  \begin{align*}
    S\big( \binprod_i (g^i, \ang{\psi^i}) \big) \circ
    S\big( \binprod_i (f^i, \ang{\phi^i}) \big)
    & = \big( g^\onen, \ang{\psi^\onen} \big) \circ \big( f^\onen, \ang{\phi^\onen} \big)\\
    & = \Big( g^\onen f^\onen, \ang{
      \pi^\onen_{\elloneelln} \cdot \si^{\elloneelln}_{g^\onen,\, f^\onen}
      }_{(\elloneelln)}
      \Big)
  \end{align*}
  where
  \[
    \pi^\onen_{\elloneelln}
    = \ga\scmap{\psi^\onen_{\elloneelln}; \ang{\phi^\onen}_{(g^\onen)^\inv\,(\elloneelln)}}
  \]
  and $\si^\elloneelln_{g^\onen,\,f^\onen}$ is the permutation as in \cref{eq:sigma-kgf} corresponding to the index $(\elloneelln)$.

  Functoriality of the Cartesian product for maps of sets implies that $h^{\onen} = g^\onen f^\onen$.  To show that $S$ is functorial, it remains to show
  \begin{equation}\label{eq:S-func-elloneelln}
    \otimes_i \big(\theta^i_{\ell_i} \cdot \si^{\ell_i}_{g^i, f^i}\big)
    =
    \ga\scmap{\psi^\onen_{\elloneelln}; \ang{\phi^\onen}_{(g^\onen)^\inv\,(\elloneelln)}}
    \cdot
    \si^{\elloneelln}_{g^\onen,\, f^\onen}
  \end{equation} 
  for each index $(\elloneelln)$.

  Since the domain of $S$ is a Cartesian product, it suffices to verify \cref{eq:S-func-elloneelln} in the cases where $(f^i,\ang{\phi^i})$ is an identity for $i \not= a$ and $(g^{i},\ang{\phi^{i}})$ is an identity for $i \not= b$, for some $1 \leq a \leq n$ and $1 \leq b \leq n$.

  If $a = b$, then \cref{eq:S-func-elloneelln} follows from relations among operations in the tensor product, \cref{definition:BVtensor}~\cref{it:BV-2} through \cref{it:BV-5}.
  If $a \not= b$ then the permutations $\si^{\ell_i}_{g^i,f^i}$ on the left hand side of \cref{eq:S-func-elloneelln} are identities.  For $a < b$, the permutation $\si^{\elloneelln}_{g^\onen,\,f^\onen}$ on the right hand side of \cref{eq:S-func-elloneelln} is also an identity and there is nothing to check.  
  For $a > b$, the permutation $\si^{\elloneelln}_{g^\onen,\,f^\onen}$ is an instance of $\xi^\otimes$ and the equality holds by \cref{definition:BVtensor}~\cref{it:interchange-relation}.
\end{proof}

\begin{proposition}\label{proposition:multilinearity-S}
  In the context of \cref{definition:S-multi}, $(S,S^2_b)$ is a strong multilinear functor.
\end{proposition}
\begin{proof}
  Functoriality of $S$ is verified in \cref{proposition:functoriality-S}.  Now we verify the multilinearity axioms of \cref{def:nlinearfunctor}.  For $n = 0$ there is nothing to check.  For $n > 0$, first note that, if any $\ang{x^i}$ is the empty tuple, then so is $\ang{x^\onen}$.  Similarly, if any $(f^i,\ang{\phi^i})$ is equal to
  \[
    \big( \varnothing\cn \ufs{0} \to \ufs{0}, \ang{} \big)\cn \ang{} \to \ang{},
  \]
  the identity morphism of the empty tuple, then $f^\onen$ is the empty morphism and $\ang{\phi^\onen}$ is also empty.  Thus $S$ satisfies the unity axiom \cref{nlinearunity} of \cref{def:nlinearfunctor}.

  The constraint unity axiom \cref{constraintunity} for $S^2_b$ holds because the permutations $\rho_{0,\hat{r}_b}$ and $\rho_{r_b,0}$ are identities.  The other three constraint axioms of \cref{def:nlinearfunctor} for $S^2_b$ follow from uniqueness of the permutations $\rho_{r_b,\hat{r}_b}$.  Since the components $S^2_b$ are determined by permutations, $S$ is a strong multilinear functor.
\end{proof}

\begin{remark}\label{remark:rho}
  In the context of \cref{definition:S-multi}, for the case $b = n$, the permutations $\rho_{r_n,\hat{r}_n}$ are identities for any $r_n$ and $\hat{r}_n$.  In particular, if $n=1$ then $S$ is the identity monoidal functor.  For $n > 1$ and $b < n$, the permutations $\rho_{r_b,\hat{r}_b}$ are generally nontrivial.
\end{remark}

\begin{lemma}\label{lemma:S-nat}
  The multilinear functors $S$ are 2-natural with respect to multifunctors and multinatural transformations
  \[
    \begin{tikzpicture}[x=40mm,y=20mm]
      \draw[0cell] 
      (0,0) node (mi) {\M_i}
      (1,0) node (ni) {\N_i.}
      ;
      \draw[1cell] 
      (mi) edge[bend left] node {H_i} (ni)
      (mi) edge[',bend right] node {K_i} (ni)
      ;
      \draw[2cell] 
      node[between=mi and ni at .5, rotate=-90, 2label={above,\theta_i}] {\Rightarrow}
      ;
    \end{tikzpicture}
  \]
\end{lemma}
\begin{proof}
  First we verify that the following diagram of permutative categories and multilinear functors commutes.
  \begin{equation}\label{eq:S-nat}
    \begin{tikzpicture}[x=40mm,y=20mm,vcenter]
      \draw[0cell] 
      (0,0) node (a) {\prod_i \bigf\M_i}
      (0,-1) node (b) {\bigf\big(\bigotimes_i \M_i \big)}
      (1,0) node (a') {\prod_i \bigf\N_i}
      (1,-1) node (b') {\bigf\big(\bigotimes_i \N_i \big)}
      ;
      \draw[1cell] 
      (a) edge node {\binprod_i \bigf H_i} (a')
      (b) edge node {\bigf(\otimes_i H_i)} (b')
      (a) edge['] node {S} (b)
      (a') edge node {S} (b')
      ;
    \end{tikzpicture}
  \end{equation}
  Commutativity on objects and morphisms follows because, by \cref{definition:free-smfun}, the strict monoidal functor $\bigf(\otimes_i H_i)$ is given by applying $\otimes_i H_i$ componentwise to tuples of objects and operations.  Therefore, both composites around the diagram above are given on objects and morphisms by the assignments
  \begin{align*}
    \big(\ang{x^1}, \ldots, \ang{x^n}\big) 
    & \mapsto \ang{(Hx)^{\onen}}\\
    \big((f^1,\ang{\phi^1}), \ldots, (f^n, \ang{\phi^n})\big)
    & \mapsto (f^{\onen}, \ang{(H\phi)^{\onen}}),
  \end{align*}
  where
  \begin{align*}
    (Hx)^\onen_\jonejn & = \otimes_i (H_ix^i_{j_i}) \andspace\\
    (H\phi)^\onen_{\konekn} & = \otimes_i (H_i \phi^i_{k_i}).
  \end{align*}
  Because $\bigf(\otimes_i H_i)$ is strict monoidal, the $b$th linearity constraint of both composites around the diagram is given by
  \[
    (\rho_{r_b,\hat{r}_b}, \ang{1}) =
    \Big( \bigf(\otimes_i H_i) \Big)(S^2_b).
  \]
  Thus the two composites in \cref{eq:S-nat} are equal as multilinear functors.

  For multinatural transformations $\theta_i \cn H_i \to K_i$, a similar analysis using \cref{definition:free-perm-multinat} shows that
  \[
    1_S * \big( \binprod_i \bigf \theta_i \big) = \bigf \big( \otimes_i \theta_i \big) * 1_S
  \]
  This completes the proof of 2-naturality for $S$.
\end{proof}

Now we define $\bigf$ on multimorphism categories.
\begin{convention}\label{convention:Fun}
  To avoid ambiguity, we let
 \[
 \bigfbar\cn \Multicat(\M,\N) \to \permcatsu(\bigf\M,\bigf\N)
 \] 
 denote the assignment $\bigf$ on 1- and 2-cells as in \cref{definition:free-smfun,definition:free-perm-multinat}.
\end{convention}

\begin{definition}\label{definition:F-multi}
  Define an assignment on multimorphism categories
  \[
    \bigf\cn \Multicat\scmap{\ang{\M};\N} \to \permcatsu\scmap{\ang{\bigf\M};\bigf\N}
  \]
  by sending data such as the following
  \[
    \begin{tikzpicture}[xscale=3,yscale=2.5,baseline={(a.base)}]
      \tikzset{0cell/.append style={nodes={scale=1}}}
      \tikzset{1cell/.append style={nodes={scale=1}}}
      \draw[0cell]
      (0,0) node (a) {\ang{\M}}
      (a)++(1,0) node (b) {\N}
      ;
      \draw[1cell]  
      (a) edge[bend left] node {H} (b)
      (a) edge[bend right] node[swap] {K} (b)
      ;
      \draw[2cell] 
      node[between=a and b at .47, shift={(0,0)}, rotate=-90, 2label={above,\theta}] {\Rightarrow}
      ;
    \end{tikzpicture}
  \]
  to the composites and whiskerings
  \begin{equation}\label{eq:FbarS}
    \begin{tikzpicture}[xscale=3,yscale=2.5,baseline={(a.base)}]
      \tikzset{0cell/.append style={nodes={scale=1}}}
      \tikzset{1cell/.append style={nodes={scale=1}}}
      \draw[0cell]
      (-1,0) node (z) {\ang{\bigf\M}}
      (0,0) node (a) {\bigf\big(\bigotimes_{i=1}^n \M_i\big)}
      (a)++(1,0) node (b) {\bigf\N.}
      ;
      \draw[1cell]  
      (a) edge[bend left] node[pos=.4] {\bigfbar H} (b)
      (a) edge[bend right] node[swap,pos=.4] {\bigfbar K} (b)
      (z) edge node {S} (a)
      ;
      \draw[2cell] 
      node[between=a and b at .47, shift={(0,0)}, rotate=-90, 2label={above,\bigfbar \theta}] {\Rightarrow}
      ;
    \end{tikzpicture}
  \end{equation}
  Multilinearity of $\bigf H = (\bigfbar H) \circ S$ follows from multilinearity of $S$ and $\bigfbar H$ being strict symmetric monoidal.  Likewise, multinaturality of $\bigf \theta = (\bigfbar \theta) * 1_S$ follows from multilinearity of $S$ and $\bigfbar\theta$ being monoidal natural.
\end{definition}

\begin{remark}\label{remark:lex-order}
  In our definition of the functor $S$, we consistently use $\tensor$ of profiles from \cref{definition:xiotimes}, with reverse lexicographic ordering, to define
  \begin{enumerate}
  \item\label{it:rmk-xonen} $\ang{ x^{\onen} }$ in \cref{eq:xonen},
  \item\label{it:rmk-fonen} $f^{\onen}$ in \cref{eq:fonen}, and
  \item\label{it:rmk-phionen} $\ang{ \phi^{\onen}}$ in \cref{eq:phionen}.
  \end{enumerate}
  This is convenient because the indices $\jonejn$ and $\konekn$ in those definitions iterate in left-to-right order.

  However, we can also use the other choice, namely, $\tensor^\transp$ in \cref{definition:xiotimes}, which corresponds to lexicographic ordering.
  In other words, in \cref{it:rmk-xonen} above we can redefine $\ang{ x^{\onen} }$ as the iterated $\tensor^\transp$-product of the tuples $\ang{x^i}$, and likewise for \cref{it:rmk-fonen,it:rmk-phionen}.
  With these consistent changes, the results about $S$ in \cref{sec:multimorphism-assignment}, the definition of $\bigf$ in \cref{eq:FbarS}, and the results in later sections are also valid.
\end{remark}

\section{Non-Symmetric \texorpdfstring{$\Cat$}{Cat}-Multifunctoriality of \texorpdfstring{$\bigf$}{F}}\label{sec:cat-multi-f}

In this section we verify the non-symmetric $\Cat$-multifunctoriality axioms for $\bigf$, from  \cref{def:enr-multicategory-functor}~\cref{it:non-symm-multifunctor}.
\begin{theorem}\label{theorem:F-multi}
  In the context of \cref{definition:F-multi}, $\bigf$ is a non-symmetric $\Cat$-multifunctor.
\end{theorem}
\begin{proof}
  To verify the axiom for units, recall from \cref{remark:rho} that $S$ is the identity monoidal functor if $n = 1$.  Since $\bigfbar$ is functorial, we have
  \[
    \bigf(1_\M) = 1_{\bigf \M}
  \]
  for each small multicategory $\M$.
  
  To verify the composition axiom, suppose given
  \begin{align*}
    H_a & \in \Multicat\scmap{\ang{\M_a};\M'_a} \forspace 1 \leq a \leq n, \andspace\\
    H' & \in \Multicat\scmap{\ang{\M'};\M''}.
  \end{align*} 
  The two multilinear functors
  \[
    \bigf\big( \ga\scmap{H';\ang{H}} \big)
    \andspace
    \ga\scmap{\bigf H'; \ang{\bigf H}}
  \]
  are given by the two composites around the boundary in the following diagram, where the unlabeled isomorphisms are given by reordering terms.
  \[
    \begin{tikzpicture}[x=30mm,y=20mm]
      \draw[0cell] 
      (0,0) node (t1) {\prod_{a,b} \bigf\M_{a,b}}
      (t1)++(0,-2) node (t2) {\bigf\big(\bigotimes_{a,b} \M_{a,b}\big)}
      (t2)++(1,0) node (t3) {\bigf\big(\bigotimes_{a} \bigotimes_{b} \M_{a,b}\big)}
      (t3)++(1.3,0) node (t4) {\bigf\big(\bigotimes_{a} \M'_a \big)}
      (t4)++(1,0) node (t5) {\bigf \M''}
      (t1)++(1,0) node (b1) {\prod_a \prod_b \bigf\M_{a,b}}
      (b1)++(0,-1) node (b2) {\prod_a \bigf\big(\bigotimes_b \M_{a,b}\big)}
      (b2)++(1.3,0) node (b3) {\prod_a \bigf \M'_{a}}
      ;
      \draw[1cell] 
      (t1) edge node {S} (t2)
      (t2) edge node {\iso} (t3)
      (t3) edge node {\bigfbar(\otimes_a H_a)} (t4)
      (t4) edge node {\bigfbar H'} (t5)
      (t1) edge node {\iso} (b1)
      (b1) edge node {\binprod_a S} (b2)
      (b2) edge node {\binprod_a \bigfbar(H_a)} (b3)
      (b3) edge node {S} (t4)
      (b2) edge node {S} (t3)
      ;
    \end{tikzpicture}
  \]
  
  In the above diagram, the two composites around the middle rectangle are equal as multilinear functors by naturality of $S$ (\cref{lemma:S-nat}) with respect to the multifunctors $H_a$.
  The rectangle at left commutes, as a diagram of underlying functors, by associativity of the products \cref{eq:xonen} and \cref{eq:phionenkonekn}.
  The linearity constraints for the composites around the rectangle at left are $(\rho, \ang{1})$ for the same permutation $\rho$ by the definition of $S^2_b$, \cref{eq:S2b}, and, in the case of the top right composite, the definition of $S$ on morphisms, \cref{eq:Sfi}, and the definition of composition, \cref{eq:FM-comp}.
  A similar diagram for multinatural transformations $H_a \to K_a$ and $H'_a \to K'_a$ commutes by the 2-naturality of $S$.
\end{proof}

\begin{example}[Non-Symmetry of $\bigf$]\label{example:F-nonsymm}
  Suppose given a permutation $\si \in \Si_n$.  The following diagram for 
  compatibility of $\bigf$ with the action of $\si$ generally does \emph{not} commute for $n \ge 2$.
  \begin{equation}\label{eq:Fnonsymm}
    \begin{tikzpicture}[x=40mm,y=-20mm,vcenter]
      \draw[0cell] 
      (0,0) node (a) {\prod_{i=1}^n \bigf\M_i}
      (0,-1) node (a') {\prod_{i=1}^n \bigf\M_{\si(i)}}
      (a)++(1,0) node (b) {\bigf\big( \bigotimes_{i=1}^n \M_i \big)}
      (a')++(1,0) node (b') {\bigf\big( \bigotimes_{i=1}^n \M_{\si(i)} \big)}
      ;
      \draw[1cell] 
      (a') edge['] node {\si} (a)
      (b') edge['] node {\bigf(\si)} (b)
      (a) edge[transform canvas={shift={(0,0pt)}}] node {S} (b)
      (a') edge[transform canvas={shift={(0,0pt)}}] node {S} (b')
      ;
    \end{tikzpicture}
  \end{equation}
  Indeed, suppose $n = 2$ with $\si$ the nontrivial transposition and consider
  \begin{align*}
    \ang{x^1} & = (x^1_1, x^1_2, x^1_3) \in \bigf\M_1\\
    \ang{x^2} & = (x^2_1, x^2_2) \in \bigf\M_2.
  \end{align*}
  Then the assignments along the top and right of \cref{eq:Fnonsymm} are the following, respectively, 
  \begin{align*}
    \big( \ang{x^2} , \ang{x^1}  \big)
    & \mapsto
      \big( x^{21}_{11}\,,\, x^{21}_{21}\,,\, x^{21}_{12}\,,\, x^{21}_{22}\,,\, x^{21}_{13}\,,\, x^{21}_{23} \big)\nonumber\\
    & \mapsto
      \big( x^{12}_{11}\,,\, x^{12}_{12}\,,\, x^{12}_{21}\,,\, x^{12}_{22}\,,\, x^{12}_{31}\,,\, x^{12}_{32} \big).
  \end{align*} 
  On the other hand, the assignments along the left and bottom of \cref{eq:Fnonsymm} are the following, respectively,
  \begin{align*}
    \big( \ang{x^2} , \ang{x^1}  \big)
    & \mapsto
      \big( \ang{x^1} , \ang{x^2}  \big)\nonumber\\
    & \mapsto
      \big( x^{12}_{11}\,,\, x^{12}_{21}\,,\, x^{12}_{31}\,,\, x^{12}_{12}\,,\, x^{12}_{22}\,,\, x^{12}_{32} \big).
  \end{align*}
  Thus the composites around \cref{eq:Fnonsymm} differ by a generally nontrivial permutation.
\end{example}

\section{Two Transformations}\label{sec:two-transf}
Recall from \cref{definition:EndC-basepoint} that the endomorphism multicategory $\End(\C)$ associated to a permutative category $\C$ defines a 2-functor
\[
\End\cn \permcatsu \to \Multicat.
\]
Throughout this section we let $\bige = \End$.  We recall from \cite{johnson-yau-permmult} two constructions that will be used to develop componentwise multinatural weak equivalences between the composites $\bige\bigf$, $\bigf\bige$, and the respective identity multifunctors.

\subsection*{Comparing $\bige\bigf$ and the Identity}
\begin{definition}\label{definition:eta}
  Suppose $\M$ is a small multicategory.
  Define a component
  \[
  \eta = \eta_{\M}\cn \M \to \bige\bigf\M
  \]
  as follows.
  For an object $w \in \M$ and an operation $\phi \in \M\mmap{y;\ang{x}}$, let $(w)$ and $(\phi)$ denote the corresponding length-1 sequences.
  For each $r \ge 0$, let $\iota_r \cn \ufs{r} \to \ufs{1}$ be the unique map of finite sets.
  Then $\eta = \eta_\M$ is the following assignment:
  \begin{align*}
    \eta w & = (w) \forspace w \in \M, \andspace\\
    \eta \phi & = (\iota_r, (\phi)) \cn \ang{x} \to (y)
  \end{align*}
  where $\phi \in \M\mmap{y;\ang{x}}$ and $|\ang{x}| = r$.
  Note that $\ang{x}$ is an $r$-fold concatenation of length-1 sequences $(x_i)$ and the morphism $(\iota_r, (\phi))$ in $\bigf\M$ is an $r$-ary operation in $\bige\bigf\M$.
  \Cref{lemma:eta-mnat} shows that each $\eta_\M$ is multifunctorial and that the components are $\Cat$-multinatural.
\end{definition}

\begin{lemma}\label{lemma:eta-mnat}
  The components $\eta_\M$ of \cref{definition:eta} define a $\Cat$-multinatural transformation
  \[
  \eta\cn 1_{\Multicat} \to \bige\bigf.
  \]
\end{lemma}
\begin{proof}
  To check multifunctoriality of each component $\eta = \eta_{\M}$, first note that $\eta$ preserves unit operations because $\iota_{1}$ is the identity on $\ufs{1}$. 
  For compatibility with the symmetric group actions, first recall 
  that the symmetry $\xi$ in $\bigf\M$, \cref{eq:FM-xi}, has the form $(\tau,\ang{1})$ and the symmetric group action in an endomorphism multicategory is given by permuting input objects (\cref{definition:EndC-basepoint}).
  Thus, for an operation $\phi \in \M\scmap{\ang{x};y}$ and a permutation $\si \in \Si_r$, where $r = |\ang{x}|$,
  the operation
  \[
    (\eta \phi) \cdot \sigma \cn \ang{x}\si \to (y) \inspace \bige\bigf\M
  \]
  is given by the composite
  $( \iota_r , (\phi) ) \circ ( \sigma , \ang{1} )$.
  We have
  \[
    ( \iota_r , (\phi) ) \circ ( \sigma , \ang{1} )
    = ( \iota_r , (\phi \cdot \sigma) )
    = \eta(\phi \cdot \sigma),
  \]
  where the first equality follows from composition \cref{eq:FM-comp} in $\bigf\M$ and right unity \cref{enr-multicategory-right-unity} in $\M$.

  For compatibility with composition, suppose given
  \[
  \psi \in \M\mmap{x'';\ang{x'}} \wherespace |\ang{x'}| = s, \andspace
  \]
  \[
  \phi_j \in \M\mmap{x'_j;\ang{x_j}} \wherespace |\ang{x_j}| = r_j \forspace
  j \in \ufs{s}.
  \]
  Let $r = \sum_j r_j$.
  Then the composite in $\bige\bigf\M$ of
  \[
  \eta \psi = (\iota_s, (\psi)) \andspace \ang{\eta \phi_j}_{j=1}^s = \ang{(\iota_{r_j}, (\phi_j))}_j
  \]
  is given by composing the morphisms
  \begin{equation}\label{eq:eta-multinat-comp}
  (\iota_s, (\psi)) \andspace \bigoplus_{j \in \ufs{s}} (\iota_{r_j}, \phi_j) = \big(\oplus_{j \in \ufs{s}} \iota_{r_j}, \ang{\phi}\big)
  \end{equation}
  in $\bigf\M$.
  For $g = \iota_s$ and $f = \oplus_j \iota_{r_j}$, the permutation $\sigma^1_{g,f}$ of \cref{eq:sigma-kgf} is the identity on $\ufs{r}$.
  Therefore the composite of the morphisms \cref{eq:eta-multinat-comp} is
  \[
  \big(\iota_s \circ (\oplus_j \iota_{r_j}), \ga\scmap{\psi;\ang{\phi}}\big)
  = \big(\iota_r, \ga\scmap{\psi;\ang{\phi}}\big) 
  = \eta\ga\scmap{\psi; \ang{\phi}}.
  \]

  Multinaturality of $\eta$ with respect to multifunctors 
  \[
    H \cn \bigotimes_{a = 1}^n \M_a \to \N
  \]
  is given by commutativity of the following outer diagram.
  \begin{equation}\label{eq:eta-mnat}
    \begin{tikzpicture}[x=50mm,y=30mm,scale=.8,vcenter]
      \draw[0cell=.9] 
      (0,0) node (a) {\bigotimes_a \M_a}
      (a)++(1,1) node (b) {\bigotimes_a \bige\bigf\M_a}
      (b)++(1,-1) node (c) {\bige\bigf\N}
      (a)++(1,-1) node (b') {\N}
      (a)++(1,.5) node (a') {\bigwedge_a \bige\bigf\M_a}
      (a')++(0,-.5) node (a'') {\bige\bigf\big(\bigotimes_a \M_a)}
      ;
      \draw[1cell=.9] 
      (a) edge node {\otimes_a \eta_{\M_a}} (b)
      (b) edge node {(\bige\bigf)H} (c)
      (a) edge['] node {H} (b')
      (b') edge['] node {\eta_{\N}} (c)
      (b) edge node {\varpi} (a')
      (a') edge node {\bige S} (a'')
      (a'') edge node {\bige(\bigfbar H)} (c)
      ;
    \end{tikzpicture}
  \end{equation} 
  In the above diagram, the multifunctor $\varpi$ is the $n$-variable version of the universal morphism 
\[\varpi_{\M,\N} \cn \M \otimes \N \to \M \sma \N\]
in \cref{MsmashN}.  The multifunctor $\bige S$ is the image of
  \[
    S\cn \prod_a \bigf\M_a \to \bigf\big(\bigotimes_a \M_a\big) \inspace \permcatsu\scmap{\ang{\bigf\M};\bigf\big(\bigotimes_a \M_a\big)}
  \]
  under the isomorphism
  \[
    \bige\cn \permcatsu\scmap{\ang{\bigf\M};\bigf\big(\bigotimes_a \M_a\big)}
    \fto{\iso}
    \Multicat_*\scmap{\ang{\bige\bigf\M};\bige\bigf\big(\bigotimes_a \M_a\big)}
  \]
  from \cref{proposition:n-lin-equiv}.  Thus the inner triangular region commutes by functoriality of $\bige$ and the equality $(\bige\bigf)H = \bige(\bigf H)$.  The remaining region commutes because both composites around its boundary have underlying assignments
  \begin{align*}
    \otimes_a x_a & \mapsto \big(H(\otimes_a x_a)\big) \forspace x_a \in \M_a \andspace\\
    \otimes_a \phi_a & \mapsto \big(\iota_r, H(\otimes_a \phi_a)\big) \forspace \phi_a \in \M_a\scmap{\ang{x_a};y_a}, 
  \end{align*}
  where each $\phi_a$ has arity $r_a$ and $r = \binprod_a r_a$ is the arity of $\otimes_a \phi_a$.
  Multinaturality of $\eta$ with respect to multinatural transformations $\ka\cn H \to K$ follows similarly.
  This completes the proof that
  \[
  \eta\cn 1_{\Multicat} \to \bige\bigf
  \]
  is a $\Cat$-enriched multinatural transformation.
\end{proof}
\begin{remark}\label{remark:eta-monoidal-natural}
  For readers familiar with $\Cat$-enriched symmetric monoidal structure as in \cite[Sections 1.4 and~1.5]{cerberusIII}, the diagram \cref{eq:eta-mnat} reduces, in the case $n = 2$ and $H = 1$, to the axiom for $\Cat$-monoidal naturality of $\eta$.
\end{remark}

\subsection*{Comparing $\bigf\bige$ and the Identity}

\begin{definition}\label{definition:rho}
  Suppose $\C$ is a small permutative category.  Define a component symmetric monoidal functor
  \[
    \rho = \rho_\C \cn \C \to \bigf\bige\C
  \]
  by the inclusion of length-1 tuples, as follows:
  \begin{align*}
    \rho x & = (x) \andspace\\
    \rho \phi & = (1_{\ufs{1}},(\phi))
  \end{align*}
  for objects $x$ and morphisms $\phi$ in $\C$.  
  The monoidal and unit constraints are given by the following morphisms
  for objects $x$ and $x'$ in $\C$:
  \begin{align*}
    (\rho^2_{\C})_{x,x'} & = (\iota_{2}, 1_{x \oplus x'})\cn (x,x') \to (x \oplus x')
                           \andspace\\
    \rho^0_{\C} & = (\iota_0,1_e)\cn \ang{} \to (e).
  \end{align*}
  Functoriality of $\rho$ follows from the formula \cref{eq:FM-comp} for composition in $\bigf\bige\C$.  The associativity and unity axioms for $(\rho_\C,\rho^2_\C,\rho^0_\C)$ follow because the product in $\bigf\bige\C$ is given by concatenation of tuples and $e$ is a strict unit for $\C$.  The symmetry axiom follows because the symmetric group action on $\bige\C$ is given by composition with the symmetry isomorphism of $\C$ (and iterates thereof).  
\end{definition}

\begin{remark}
  One can show that the components $\rho_\C$ are multinatural with respect to multilinear functors and transformations.  However, these components are not strictly unital because the unit constraints $\rho^0$ are not identities.  Thus the components $\rho_\C$ are not 1-cells in $\permcatsu$.
  To address this, we recall the following construction, replacing a general symmetric monoidal functor with a zigzag of strictly unital symmetric monoidal functors.
\end{remark}

\begin{definition}\label{definition:Cmark}
  Define a functor
  \[
    (-)^\whisk \cn \permcat \to \permcatsu
  \]
  that adjoins a new monoidal unit, as follows.
  \begin{description}
  \item[Objects]
    Suppose $(\C,\oplus,e)$ is a permutative category.  Let $\C^\whisk$ denote the permutative category whose objects, morphisms, and symmetric monoidal structure are given by those of $\C$, together with an additional object $0$ that is a strict monoidal unit and an additional morphism $t\cn 0 \to e$.

    Let $I = \{0 \to 1\}$ denote the permutative category formed by two idempotent objects and a single morphism from the monoidal unit to the other object.  Then $\C^\whisk$ is the permutative category obtained by adjoining $I$ to $\C$ by identifying the objects $1$ and $e$.

  \item[Morphisms]
    Suppose $P \cn \C \to \D$ is a symmetric monoidal functor.  Let
    \begin{equation}\label{eq:Pwhisk}
      P^\whisk \cn \C^\whisk \to \D^\whisk
    \end{equation} 
    be the strictly unital symmetric monoidal functor that is given by $P$ on objects and morphisms of $\C$, and that sends the additional morphism $t$ in $\C^\whisk$ to the composite
    \[
      0 \fto{t} e^\D \fto{P^0} P(e^\C)
    \]
    in $\D^\whisk$.

    The monoidal constraint of $P^\whisk$ is given by that of $P$ for objects $x,y \neq 0$ in $\C$.  If either $x$ or $y$ is $0$, then $P^\whisk(0) = 0 \in \D$ and the monoidal constraint is an identity morphism.
  \end{description}

  There is a strict symmetric monoidal functor
  \begin{equation}\label{eq:retr}
    \C^\whisk \fto{r_\C} \C
  \end{equation} 
  that is the identity on objects and morphisms of $\C$ and sends the additional morphism $t$ of $\C^\whisk$ to the identity morphism of $e$ in $\C$.  We call this the \emph{retraction} for $\C^\whisk$.  We let
  \begin{equation}\label{eq:jC}
    \C \fto{j_\C} \C^\whisk
  \end{equation} 
  denote the inclusion, so that $r_\C j_\C$ is the identity on $\C$.  Note that $r_\C \dashv j_\C$ is an adjunction of underlying categories.
\end{definition}

\begin{explanation}[Composition and Monoidal Sum in $\C^\whisk$]
  The morphisms $0 \to x$ in $\C^\whisk$, for $x \neq 0$, are given by composition with $t$.  Thus, the morphism sets in $\C^\whisk$ are given as follows, for $x,y \neq 0$
  \begin{align*}
    \C^\whisk(0,0)
    & = \{1_0\}
    & \C^\whisk(0,e)
    & = \{t\}\\
    \C^\whisk(0,x)
    & = \C(e,x) \times \{t\}
    & \C^\whisk(x,0)
    & = \varnothing\\
    \C^\whisk(x,y)
    & = \C(x,y).
  \end{align*}
  The composition
  \[
    \C^\whisk(x,y) \times \C^\whisk(0,x) \to \C^\whisk(0,y)
  \]
  sends a pair of morphisms $f\cn x \to y$ and $gt\cn 0 \to e \to x$ to $(fg)t$.

  The monoidal sum $\oplus$ on morphisms in $\C^\whisk$ is defined as follows.
  \begin{enumerate}
  \item For morphisms $f\cn e \to x$ and $g\cn e \to y$ in $\C$,
    \[
      (0 \fto{t} e \fto{f} x) \oplus (0 \fto{t} e \fto{g} y)
      =
      (0 \fto{t} e \fto{1_e} e \oplus e \fto{f \oplus g} x \oplus y).
    \]
    In particular, $t \oplus t = t$.
  \item For morphisms $f \cn e \to x$ and $h\cn a \to b$ in $\C$,
    \[
      (0 \fto{t} e \fto{f} x) \oplus (a \fto{h} b)
      =
      (0 \oplus a \fto {1_a} e \oplus a \fto{f \oplus h}  x \oplus b).
    \]
    Similarly, $h \oplus (ft) = (h \oplus f)$.\dqed
  \end{enumerate}
\end{explanation}

It will be helpful to introduce additional notation for the following slight variant of the above construction $(-)^\whisk$.
The variant is used in the statement of \cref{lemma:rho-multinat} and other discussion below.
\begin{definition}\label{definition:Pmark}
  Given a symmetric monoidal functor
  \[
    (P,P^2,P^0)\cn \C \to \D,
  \]
  let $\C^\tail = \C^\whisk$ and let $P^\tail$ denote the composite of $P^\whisk$ with $r_{\D}$ shown in the following diagram.
  \begin{equation}\label{eq:Pmark}
    \begin{tikzpicture}[x=30mm,y=10mm,vcenter]
      \draw[0cell] 
      (1,0) node (ct) {\C^\tail}
      (1,-1) node (cw) {\C^\whisk}
      (2,0) node (d) {\D}
      (2,-1) node (dw) {\D^\whisk}
      ;
      \draw[1cell] 
      (ct) edge node {P^\tail} (d)
      (ct) edge[equal] node {} (cw)
      (cw) edge node {P^\whisk} (dw)
      (dw) edge['] node {r_\D} (d)
      ;
    \end{tikzpicture}
  \end{equation} 
  Thus, $P^\tail$ extends $P$ by sending the additional morphism $t$ of $\C^\tail$ to the unit constraint $P^0$.

  The construction $(-)^\tail$ provides a multifunctor
  \[
    \permcatsu \to \permcatsu
  \]
  determined by composition and whiskering with the canonical pointed multifunctors
  \[
    \bigwedge_i \bige (\C_i^\tail) \to \big(\bigwedge_i \bige\C_i\big)^\tail 
  \]
  that are strictly unital and determined by the identity on each $\bige\C_i$.
\end{definition}

\begin{lemma}\label{lemma:rho-multinat}
  Using the constructions of \cref{eq:retr} and \cref{definition:Pmark}, the components $\rho_\C$ define a zigzag of $\Cat$-multinatural transformations 
  \[
    1_{\permcatsu} \xleftarrow{r} (-)^\tail \fto{\rho^\tail} \bigf\bige.
  \]
\end{lemma}
\begin{proof}
  For a multilinear functor
  \[
    \ang{\C} \fto{P} \D,
  \]
  the left half of the diagram below commutes by definition of $(-)^\tail$ on multilinear functors $P$.
  \begin{equation}\label{eq:rhomark-mnat}
    \begin{tikzpicture}[x=50mm,y=20mm,scale=.8,vcenter]
      \draw[0cell=.9] 
      (0,0) node (a) {\prod_a^{\phantom{a}} \C_a}
      (a)++(1,0) node (a') {\prod_a^{\phantom{a}} \C_a^\tail}
      (a')++(1,0) node (b) {\prod_a^{\phantom{a}} \bigf\bige\C_a}
      (a)++(0,-1) node (d) {\D}
      (a')++(0,-1) node (d') {\D^\tail}
      (b)++(0,-1) node (e) {\bigf\bige\D}
      ;
      \draw[1cell=.9] 
      (a') edge['] node {\binprod_a r_{\C_a}} (a)
      (d') edge['] node {r_{\D}} (d)
      (a') edge node {\binprod_a \rho^\tail_{\C_a}} (b)
      (d') edge node {\rho^\tail_\D} (e)
      (a) edge['] node {P} (d)
      (b) edge node {\bigf\bige P} (e)
      (a') edge node {(P)^\tail} (d')
      ;
    \end{tikzpicture}
  \end{equation} 
  Away from the new unit objects $0 \in \C_a^\bullet$, the right half of the above diagram commutes because $\rho$ is the inclusion of length-one tuples.  Commutativity for unit objects follows because each arrow strictly preserves units.  A similar analysis applies to multilinear transformations $\theta \cn P \to Q$.
\end{proof}

\section{Equivalence of Homotopy Theories}\label{sec:hty-thy-equiv}

For permutative categories and for multicategories, we define stable equivalences via the stable equivalences on $K$-theory spectra.  We let $\SymSp$ denote the Hovey-Shipley-Smith category of symmetric spectra \cite{hss}.  We let
\[
\Kse \cn \PermCatsu \to \SymSp
\]
denote Segal's $K$-theory functor \cite{segal} that constructs a connective symmetric spectrum from each small permutative category.
See \cite[Chapters~7 and~8]{cerberusIII} for a review and further references.

\begin{definition}\label{definition:she}
  We define stable equivalences, $\cS$, of permutative categories and multicategories via the three functors
  \[
    \begin{tikzpicture}[x=30mm,y=15mm]
      \draw[0cell] 
      (0,0) node (a) {\permcat}
      (0,-1) node (b) {\Multicat}
      (1,-.5) node (c) {\permcatsu}
      (2,-.5) node (d) {\SymSp}
      ;
      \draw[1cell] 
      (a) edge node {(-)^\whisk} (c)
      (b) edge['] node {\bigf} (c)
      (c) edge node {\Kse} (d)
      ;
    \end{tikzpicture}
  \]
  as follows.
  \begin{itemize}
  \item A strictly unital symmetric monoidal functor $P$ is a \emph{stable equivalence} if $\Kse P$ is a stable equivalence of connective spectra.
    
  \item A multifunctor $H$ is a \emph{stable equivalence} if $\bigf H$ is a stable equivalence of permutative categories.

  \item A general symmetric monoidal functor $P$ is a \emph{stable equivalence} if $P^\whisk$ is a stable equivalence in $\permcatsu$.  That is, $P$ is a stable equivalence if $\Kse(P^\whisk)$ is a stable equivalence of connective spectra.  \Cref{proposition:permcat-permcatsu} below shows that when $P$ is strictly unital, then $\Kse(P^\whisk)$ is a stable equivalence of connective spectra if and only if $\Kse P$ is so.
  \end{itemize}

  Thus, the stable equivalences in each of $\permcatsu$, $\permcat$, and $\Multicat$, respectively, are reflected by $\Kse$, $(-)^\whisk$, and $\bigf$.  We let $\cS$ denote the class of stable equivalences in each case.
\end{definition}

The following result shows that the homotopy theories determined by stable equivalences in $\permcat$ and $\permcatsu$ are equivalent.  The result is well known to experts; see e.g., \cite[Theorem~3.9]{mandell_inverseK} and \cite[Theorem~2.15]{gjo-extending} for discussion of the general symmetric monoidal case.  For completeness, we present a short proof based on \cite[Theorem~1.11]{gjo-extending}.
\begin{proposition}\label{proposition:permcat-permcatsu}
  There is an adjunction $(-)^\whisk \dashv I$ that induces an equivalence of homotopy theories 
  \begin{equation}\label{eq:permcat-permcatsu-steq}
    (-)^\whisk\cn \big( \permcat, \cS \big) \lrsimadj \big( \permcatsu, \cS  \big) \cn I
  \end{equation}
  where $I$ is the inclusion.
\end{proposition}
\begin{proof}
  The unit and counit of the adjunction are given by
  \begin{align*}
    j_{\C} & \cn \C \to \C^\whisk = I(\C^\whisk) \andspace\\
    r_{\C} & \cn (I\C)^\whisk = \C^\whisk \to \C
  \end{align*}
  from \cref{eq:jC} and \cref{eq:retr}, respectively.
  Because $I$ is a subcategory inclusion, it is a \emph{map extension} in the sense of \cite[Definition~1.9]{gjo-extending}.  Moreover, each component of the counit, $r$, is a stable equivalence in $\permcatsu$ because $r_\C \dashv j_\C$ is an adjunction of underlying categories.  Thus, the adjunction $(-)^\whisk \dashv I$ satisfies the conditions of \cite[Theorem~1.11~(3)]{gjo-extending}: the left adjoint creates stable equivalences and each component of the counit is a stable equivalence.  Therefore, condition \cite[Theorem~1.11~(2)]{gjo-extending} also holds: the right adjoint $I$ creates stable equivalences and the unit is componentwise a stable equivalence in $\permcat$.  It follows that \cref{eq:permcat-permcatsu-steq} is an adjoint stable equivalence of homotopy theories.
\end{proof}

Next we recall further properties of $\eta$ and $\rho$ from \cite{johnson-yau-permmult}.
\begin{proposition}[{\cite[6.11,~6.13,~7.3]{johnson-yau-permmult}}]\label{proposition:two-transf}
  The following statements hold for a small multicategory $\M$ and a small permutative category $\C$.
  \begin{enumerate}
  \item\label{it:eta-steq} The component
    \[
      \eta \cn \M \to \bige\bigf\M
    \]
    is a stable equivalence of multicategories.
  \item\label{it:rho-rt-adj} The component
    \[
      \rho \cn \C \to \bigf\bige\C
    \]
    is a right adjoint of underlying categories.
  \end{enumerate}
\end{proposition}
\begin{remark}\label{remark:epz}
  Both statements of \cref{proposition:two-transf} are proved by using construction of strict symmetric monoidal functors \cite[6.4]{johnson-yau-permmult}
  \[
    \epz = \epz_\C \cn \bigf\bige\C \to \C
  \]
  for each small permutative category $\C$, given by the following assignments:
  \begin{align}
    \epz\ang{x} & = \bigoplus_i x_i 
                  \andspace\label{eq:xi-x}\\
    \epz(f,\ang{\phi}) & = \Big(\bigoplus_j \phi_j \Big) \circ \xi_f\label{eq:xi-fphi}
  \end{align}
  where $\ang{x}$ is an object of $\bigf\bige\C$, $(f,\ang{\phi})\cn \ang{x} \to \ang{y}$ is a morphism of $\bigf\bige\C$, and $\xi_f$ is a certain permutation of summands \cite[9.2]{johnson-yau-permmult} determined by $f$.
  The results in \cite{johnson-yau-permmult} show that $\eta$ and $\epz$ are the unit and counit of a 2-adjunction between the 2-categories $\Multicat$ and $\permcatst$, where the 1-cells are \emph{strict} monoidal functors.  Moreover, for each $\C$ the components $(\epz_\C,\rho_\C)$ are an adjunction of underlying categories.

  However, the components $\epz_\C$ are \emph{not} natural with respect to general strictly unital symmetric monoidal functors.
  More generally, the following multinaturality diagram for $\epz$ with respect to a multilinear functor
  \[
    \ang{\C} \fto{P} \D
  \]
  fails to commute unless each of the linearity constraints $P^2_a$ is an identity.
  \begin{equation}\label{eq:epz-non-multinat}
    \begin{tikzpicture}[x=50mm,y=30mm,scale=.4,vcenter]
      \draw[0cell=.8] 
      (0,0) node (a) {\prod_a \bigf\bige\C_a}
      (a)++(1,1) node (b) {\prod_a \C_a}
      (b)++(1,-1) node (c) {\D}
      (a)++(1,-1) node (b') {\bigf\bige \D}
      ;
      \draw[1cell=.9] 
      (a) edge node {\binprod_a \epz_{\C_a}} (b)
      (b) edge node {P} (c)
      (a) edge['] node {\bigf\bige P} (b')
      (b') edge['] node {\epz_{\D}} (c)
      ;
    \end{tikzpicture}
  \end{equation} 
  Thus the components $\epz$ do not provide a counit for $\bigf$ and $\bige$ to be adjunction between $\Multicat$ and $\permcatsu$, either as 2-categories or as $\Cat$-multicategories.
\end{remark}

\begin{definition}\label{definition:alg-she}
  Suppose $\cO$ is a small non-symmetric $\Cat$-multicategory.  A \emph{non-symmetric $\cO$-algebra} in a $\Cat$-multicategory $\M$ is a non-symmetric $\Cat$-multifunctor
  \[
    \cO \to \M.
  \]
  The morphisms of non-symmetric $\cO$-algebras are $\Cat$-multinatural transformations.  The category of non-symmetric $\cO$-algebras and their morphisms in $\M$ is denoted $\M^\cO$.
  Considering the underlying 1-category of $\M$, suppose $(\M,\cW)$ is a relative category.  We define $\cW^\cO$ as the subcategory of $\M^\cO$ that contains all the non-symmetric $\cO$-algebras and the morphisms that are componentwise in $\cW$.  We consider the pair $(\M^\cO, \cW^\cO)$ as a relative category.
\end{definition}

Now we come to the proof of \cref{theorem:alg-hty-equiv}.
\begin{proof}[Proof of \cref{theorem:alg-hty-equiv}]
  The $\Cat$-multifunctoriality of $\bige$ and $\bigf$ (non-symmetric in the latter case) is described in \cref{sec:permcat,sec:cat-multi-f}, respectively.
  By \cref{lemma:eta-mnat,proposition:two-transf}~\cref{it:eta-steq}, $\eta$ provides a natural stable equivalence
  \begin{equation}\label{eq:1-to-EF}
    \eta^\cO \cn 1 \to (\bige^\cO)(\bigf^\cO).
  \end{equation} 
  Naturality of $\eta^\cO$ with respect to algebra morphisms follows from $\Cat$-multinaturality of $\eta$ in \cref{lemma:eta-mnat}.

  For the other composite, $\Cat$-multinaturality of the zigzag
  \[
    1 \xleftarrow{r} (-)^\tail \fto{\rho^\tail} \bigf\bige
  \]
  in \cref{lemma:rho-multinat} induces a zigzag of natural transformations between endofunctors of $\big( \permcatsu \big)^\cO$
  \begin{equation}\label{eq:zigzag-FE}
    1 \leftarrow (-)^\tail \rightarrow (\bigf^\cO)(\bige^\cO).
  \end{equation}
  For each permutative category $\C$, the adjunction of underlying categories $r_\C \dashv j_\C$ implies that $r_\C$ is a stable equivalence in $\permcatsu$ and $j_\C$ is a stable equivalence in $\permcat$.
  Furthermore, each $\rho_\C$ is a stable equivalence in $\permcat$ because it is a right adjoint of underlying categories, by \cref{proposition:two-transf}~\cref{it:rho-rt-adj}.
  Therefore, by the 2-out-of-3 property for stable equivalences and the factorization
  \[
    \rho_\C = \rho^\tail_\C \circ j_\C \cn \C \to \bigf\bige\C,
  \]
  each $\rho^\tail_\C$ is a stable equivalence in $\permcat$.  By
  \cref{proposition:permcat-permcatsu}, each $\rho^\tail_\C$ is therefore
  also a stable equivalence in $\permcatsu$.
  This implies that \cref{eq:zigzag-FE} is a zigzag of stable equivalences.

  Therefore, by \cref{gjo29}, $\bigf^\cO$ and $\bige^\cO$ are equivalences of homotopy theories.  This completes the proof.
\end{proof}

\section{Application to Ring Categories}\label{sec:application}

As a consequence of \cref{theorem:alg-hty-equiv}, the functors $\bigf^\cO$ and $\bige^\cO$ induce equivalences of homotopy theories of non-symmetric $\cO$-algebras.  This section gives an application to ring categories, which are non-symmetric algebras over the $\Cat$-enriched associative operad, $\As$, by \cite[11.2.16]{cerberusIII}.  We recall the essential definitions.

\begin{definition}[\cite{elmendorf-mandell}]\label{def:ringcat}
A \emph{ring category} is a tuple
\[\big(\C,(\oplus,\zer,\xiplus),(\otimes,\tu),(\fal,\far)\big)\]
consisting of the following data.
\begin{description}
\item[The Additive Structure] 
$(\C,\oplus,\zer,\xiplus)$ is a permutative category. 
\item[The Multiplicative Structure] 
$(\C,\otimes,\tu)$ is a strict monoidal category.
\item[The Factorization Morphisms] 
$\fal$ and $\far$ are natural transformations
\begin{equation}\label{ringcatfactorization}
\begin{tikzcd}[row sep=tiny,column sep=huge]
(A \otimes C) \oplus (B \otimes C) \ar{r}{\fal_{A,B,C}} & (A \oplus B) \otimes C\\
(A \otimes B) \oplus (A \otimes C) \ar{r}{\far_{A,B,C}} & A \otimes (B \oplus C)
\end{tikzcd}
\end{equation}
for objects $A,B,C\in\C$, which are called the \emph{left factorization morphism} and the \emph{right factorization morphism}, respectively.
\end{description}
We often abbreviate $\otimes$ to concatenation, with $\tensor$ always taking precedence over $\oplus$ in the absence of clarifying parentheses.  The subscripts in $\xiplus$, $\fal$, and $\far$ are sometimes omitted. 

The above data are required to satisfy the following seven axioms for all objects $A$, $A'$, $A''$, $B$, $B'$, $B''$, $C$, and $C'$ in $\C$.  Each diagram is required to be commutative.
\begin{description}
\item[The Multiplicative Zero Axiom] 
\[\begin{tikzcd}[column sep=large]
\boldone \times \C \ar{d}[swap]{\zer \times 1_{\C}} \ar{r}{\iso} & \C \ar{d}{\zer} & \C \times \boldone \ar{l}[swap]{\iso} \ar{d}{1_{\C} \times \zer}\\
\C \times \C \ar{r}{\otimes} & \C & \C \times \C \ar{l}[swap]{\otimes}
\end{tikzcd}\]
In this diagram, the top horizontal isomorphisms drop the $\boldone$ argument.  Each $\zer$ denotes the constant functor at $\zer \in \C$ and $1_\zer$.
\item[The Zero Factorization Axiom] 
\[\begin{aligned}
\fal_{\zer,B,C} &= 1_{B \otimes C} &\qquad \far_{\zer,B,C} &= 1_{\zer}\\
\fal_{A,\zer,C} &= 1_{A \otimes C} & \far_{A,\zer,C} &= 1_{A \otimes C}\\
\fal_{A,B,\zer} &= 1_{\zer} & \far_{A,B,\zer} &= 1_{A \otimes B}
\end{aligned}\]
\item[The Unit Factorization Axiom] 
\[\begin{split}
\fal_{A,B,\tu} &= 1_{A \oplus B}\\
\far_{\tu,B,C} &= 1_{B \oplus C}
\end{split}\]
\item[The Symmetry Factorization Axiom] 
\[\begin{tikzcd}
AC \oplus BC \ar{d}[swap]{\xiplus} \ar{r}{\fal} & (A \oplus B)C \ar{d}{\xiplus 1_C}\\
BC \oplus AC \ar{r}{\fal} & (B \oplus A)C
\end{tikzcd}\qquad
\begin{tikzcd}
AB \oplus AC \ar{d}[swap]{\xiplus} \ar{r}{\far} & A(B \oplus C) \ar{d}{1_A \xiplus}\\
AC \oplus AB \ar{r}{\far} & A(C \oplus B)
\end{tikzcd}\]
\item[The Internal Factorization Axiom]  
\[\begin{tikzcd}[cells={nodes={scale=.8}}]
AB \oplus A'B \oplus A''B \ar{d}[swap]{1 \oplus \fal} \ar{r}{\fal \oplus 1} & (A \oplus A')B \oplus A''B \ar{d}{\fal}\\
AB \oplus (A' \oplus A'')B \ar{r}{\fal} & (A \oplus A' \oplus A'')B
\end{tikzcd}\qquad
\begin{tikzcd}[cells={nodes={scale=.8}}]
AB \oplus AB' \oplus AB'' \ar{d}[swap]{1 \oplus \far} \ar{r}{\far \oplus 1} & A(B \oplus B') \oplus AB'' \ar{d}{\far}\\
AB \oplus A(B' \oplus B'') \ar{r}{\far} & A(B \oplus B' \oplus B")
\end{tikzcd}\]
\item[The External Factorization Axiom] 
\[\begin{tikzcd}[column sep=large,cells={nodes={scale=.8}}]
ABC \oplus A'BC \ar{d}[swap]{\fal_{AB,A'B,C}} \ar{r}{\fal_{A,A',BC}} & (A \oplus A')BC \ar[equal]{d}\\
(AB \oplus A'B)C \ar{r}{\fal_{A,A',B} 1_C} & (A \oplus A')BC
\end{tikzcd}
\qquad
\begin{tikzcd}[column sep=large,cells={nodes={scale=.8}}]
ABC \oplus AB'C \ar{d}[swap]{\far_{A,BC,B'C}} \ar{r}{\fal_{AB,AB',C}} & (AB \oplus AB')C \ar{d}{\far 1_C}\\
A(BC \oplus B'C) \ar{r}{1_A \fal_{B,B',C}} & A(B \oplus B')C
\end{tikzcd}\]
\[\begin{tikzcd}[column sep=huge]
ABC \oplus ABC' \ar{d}[swap]{\far_{A,BC,BC'}} \ar{r}{\far_{AB,C,C'}} & AB(C \oplus C') \ar[equal]{d}\\
A(BC \oplus BC') \ar{r}{1_A \far_{B,C,C'}} & AB(C \oplus C')
\end{tikzcd}\]
\item[The 2-By-2 Factorization Axiom]
\[\begin{tikzpicture}[xscale=3,yscale=1,baseline={(x2.base)}]
\def\h{1} \def\v{-1}
\draw[0cell=.8] 
(0,0) node (x0) {A(B \oplus B') \oplus A'(B \oplus B')}
(x0)++(-\h,\v) node (x1) {AB \oplus AB' \oplus A'B \oplus A'B'}
(x1)++(1.3*\h,\v) node (x2) {(A \oplus A')(B \oplus B')}
(x1)++(0,2*\v) node (x3) {AB \oplus A'B \oplus AB' \oplus A'B'}
(x3)++(\h,\v) node (x4) {(A \oplus A')B \oplus (A \oplus A')B'}
;
\draw[1cell=.9] 
(x1) edge node[pos=.4] {\far \oplus \far} (x0)
(x0) edge node {\fal} (x2)
(x1) edge node[swap] {1 \oplus \xiplus \oplus 1} (x3)
(x3) edge node[swap,pos=.4] {\fal \oplus \fal} (x4)
(x4) edge node[swap] {\far} (x2)
; 
\end{tikzpicture}\]
\end{description}
This finishes the definition of a ring category.
\end{definition}

Recall, e.g., from \cite[Section~11.1]{cerberusIII} the $\Cat$-enriched associative operad $\As$ detects monoid structures.  Since $\As$ is the free symmetric operad on the terminal non-symmetric operad, which we denote $\As'$, monoid structures are also detected by non-symmetric algebras over $\As'$.  This yields the following equivalent formulation of \cite[11.2.16]{cerberusIII}.
\begin{theorem}[{\cite[11.2.16]{cerberusIII}}]\label{theorem:III.11.2.16}
  For each small permutative category $\C$, there is a canonical bijective correspondence between
  \begin{itemize}
  \item ring category structures on $\C$ and
  \item non-symmetric $\Cat$-enriched multifunctors
    \[
      H \cn \As' \to \permcatsu \stspace H(*) = \C.
    \]
  \end{itemize}
\end{theorem}

Defining the morphisms of ring categories via morphisms of non-symmetric $\As'$-algebras, \cref{theorem:III.11.2.16} yields the following application of \Cref{theorem:alg-hty-equiv}.
\begin{corollary}\label{corollary:En-alg-hty-equiv}
  The $\Cat$-multifunctors (non-symmetric in the case of $\bigf$)
  \[
    \bigf \cn \Multicat \lradj \permcatsu \cn \bige,
  \]
  induce an equivalence of homotopy theories between
  associative monoids in $\Multicat$ and ring categories (\Cref{def:ringcat}).
\end{corollary}